\documentclass[11pt]{article}

\usepackage[english]{babel}
\usepackage[a4paper]{geometry}
\usepackage[utf8]{inputenc}
\usepackage[T1]{fontenc}
\usepackage[all,cmtip]{xy}				
\usepackage{amsmath,amssymb,amsthm}	
\usepackage[mathscr]{eucal}			
\usepackage{amsmath,amssymb}
\usepackage{amsthm}
\usepackage{amsfonts}
\usepackage{enumerate}
\usepackage{mathabx}
\usepackage[all,cmtip]{xy}
\usepackage{amscd}
\usepackage[svgnames]{xcolor}
\usepackage{listings}
\usepackage{enumitem}

\theoremstyle{plain} 
\newtheorem{thm}{Theorem}[section]
\newtheorem{cor}[thm]{Corollary} 
\newtheorem{lemma}[thm]{Lemma} 
\newtheorem{prop}[thm]{Proposition}

\theoremstyle{definition} 
\newtheorem{dfn}[thm]{Definition}
\newtheorem{rmk}[thm]{Remark} 

\newtheorem*{rw}{Related works}
\newtheorem*{plan}{Plan of the paper}
\newtheorem*{ack}{Acknowledgements}

\newtheorem*{rank3}{Rank 3}
\newtheorem*{rank4}{Rank 4}

\DeclareMathOperator{\C}{\mathbb{C}}
\DeclareMathOperator{\p}{\mathbb{P}}
\DeclareMathOperator{\Q}{\mathbb{Q}}		
\DeclareMathOperator{\Z}{\mathbb{Z}}

\DeclareMathOperator{\N}{\mathbb{N}}

\DeclareMathOperator{\A}{\mathcal{A}}
\DeclareMathOperator{\D}{\text{D}^b}
\DeclareMathOperator{\G}{\text{Gr}}
\DeclareMathOperator{\CG}{\text{CGr}}

\DeclareMathOperator{\ch}{\text{ch}}
\DeclareMathOperator{\td}{\text{td}}
\DeclareMathOperator{\KT}{\emph{K}(\mathcal{A}_X)_{\text{top}}}

\newcommand\blfootnote[1]{%
  \begingroup
  \renewcommand\thefootnote{}\footnote{#1}%
  \addtocounter{footnote}{-1}%
  \endgroup
}
	
\title{\textbf{On the double EPW sextic associated to a Gushel-Mukai fourfold}}
\author{Laura Pertusi}
\date{ }

\begin{document}
\pagestyle{plain}
\maketitle

\begin{abstract}
In analogy to the case of cubic fourfolds, we discuss the conditions under which the double cover $\tilde{Y}_A$ of the EPW sextic hypersurface associated to a Gushel-Mukai fourfold is birationally equivalent to a moduli space of (twisted) stable sheaves on a K3 surface. In particular, we prove that $\tilde{Y}_A$ is birational to the Hilbert scheme of two points on a K3 surface if and only if the Gushel-Mukai fourfold is Hodge-special with discriminant $d$ such that the negative Pell equation $\mathcal{P}_{d/2}(-1)$ is solvable in $\Z$.
\end{abstract}

\blfootnote{\textup{2010} \textit{Mathematics Subject Classification}: \textup{14J35}, \textup{14J28}, \textup{18E30}.}

\section{Introduction}

The geometry of Gushel-Mukai (GM) varieties has been recently studied by Debarre and Kuznetsov in \cite{DK1}, \cite{DK2}, and from a categorical point of view by Kuznetsov and Perry in \cite{KP}. Of particular interest is the case of GM fourfolds, which are smooth intersections of dimension four of the cone over the Grassmannian $\G(2,5)$ with a quadric hypersurface in a eight-dimensional linear space over $\C$. Indeed, these Fano fourfolds have a lot of similarities with cubic fourfolds; for instance, it is unknown if the very general GM fourfold is irrational, even if there are rational examples (see \cite[Section 4]{R}, \cite[Section 3]{P} and \cite[Section 7]{DIM}). 

In \cite{DIM} Debarre, Iliev and Manivel investigated the period map and the period domain of GM fourfolds, in analogy to the work done by Hassett for cubic fourfolds. In particular, they proved that period points of \emph{Hodge-special} GM fourfolds (see Definition \ref{defHodge-spe}) form a countable union of irreducible divisors in the period domain, depending on the discriminant of the possible labellings (see Section 2.3). It is not difficult to check that the discriminant of a Hodge-special GM fourfold is an integer $\equiv 0,2$ or $4 \pmod 8$ (see \cite{DIM}, Lemma 6.1). Furthermore, the non-special cohomology of a Hodge-special GM fourfold $X$ is Hodge isometric (up to a Tate twist) to the degree-two primitive cohomology of a polarized K3 surface if and only if the discriminant $d$ of $X$ satisfies also the following numerical condition:
\begin{equation*}\tag{$\ast\ast$}
8 \nmid d \text{ and the only odd primes which divide }d \text{ are } \equiv 1 \pmod 4.
\end{equation*} 

The first result of this paper is a generalization of the previous property to the twisted case, as done by Huybrechts for cubic fourfolds in \cite{Huy}. 
\begin{thm}
\label{thm3intro}
A GM fourfold $X$ has an associated twisted K3 surface in the cohomological sense (see Definition \ref{deftwist}) if and only if the discriminant $d$ of $X$ satisfies
\begin{equation*}\tag{$\ast\ast'$}
d=\prod_{i}p_i^{n_i} \text{ with } n_i \equiv 0 \pmod 2 \text{ for } p_i \equiv 3 \pmod 4.
\end{equation*}
\end{thm} 

On the other hand, a general GM fourfold $X$ has an associated hyperk\"ahler variety, as cubic fourfolds have their Fano variety of lines. Indeed, $X$ determines a triple $(V_6,V_5,A)$ of Lagrangian data, where $V_6 \supset V_5$ are six and five-dimensional vector spaces, respectively, and $A \subset \bigwedge^3V_6$ is a Lagrangian subspace with respect to the symplectic structure induced by the wedge product, with no decomposable vectors (see \cite[Theorem 3.16]{DK1}). Conversely, it is possible to reconstruct an ordinary and a special GM variety from a Lagrangian data having $A$ without decomposable vectors (see \cite[Theorem 3.10 and Proposition 3.13]{DK1}). The data of $A$ determines a stratification in subschemes of the form $Y_A^{\geq 3} \subset Y_A^{\geq 2} \subset Y_A^{\geq 1} \subset \p(V_6)$, where $Y_A^{\geq 1}$ is a Eisenbud-Popescu-Walter (EPW) sextic hypersurface (see Section 2.2). Moreover, if $Y_A^{\geq 3}$ is empty, then the double cover $\tilde{Y}_A$ of the EPW sextic is a hyperk\"ahler fourfold deformation equivalent to the Hilbert scheme of length-two subschemes on a K3 surface. Actually, in order to guarantee the smoothness of $\tilde{Y}_A$, it is enough to avoid the divisor $\mathcal{D}_{10}''$ in the period domain by \cite[Remark 5.29]{DK2}.     

The second main result is the following theorem, whose analogue for cubic fourfolds was proven by Addington in \cite{Add}. Let $\lambda_1$ and $\lambda_2$ be the classes in the topological K-theory of a GM fourfold defined in \eqref{basis}. 

\begin{thm}
\label{thm1}
Let $X$ be a Hodge-special GM fourfold such that $Y_A^{\geq 3}=\emptyset$. Consider the following propositions:
\begin{enumerate}
\item[$(\emph{a})$] $X$ has discriminant $d$ satisfying $(\ast\ast)$;
\item[$(\emph{b})$] The trascendental lattice $T_X$ is Hodge isometric to the trascendental lattice $T_S(-1)$ for some K3 surface $S$, or equivalently, there is a hyperbolic plane $U=\langle \kappa_1,\kappa_2 \rangle$ primitively embedded in the algebraic part of the Mukai lattice;
\item[$(\emph{c})$] $\tilde{Y}_A$ is birational to a moduli space of stable sheaves on $S$.
\end{enumerate}
Then we have $(\emph{a})$ implies $(\emph{b})$, and $(\emph{b})$ is equivalent to $(\emph{c})$. 

Moreover, $(\emph{b})$ implies $(\emph{a})$ if either $H^{2,2}(X,\Z)$ has rank-$3$, or there is an element $\tau$ in the hyperbolic plane $U$ such that $\langle \lambda_1, \lambda_2, \tau \rangle$ has discriminant $\equiv 2$ or $4 \pmod 8$.
\end{thm}  
\noindent In Section 3.3 we discuss a counterexample showing that the reverse implication of Theorem \ref{thm1} does not hold in full generality. More precisely, we show that there are GM fourfolds satisfying condition (b), but without a Hodge-associated K3 surface. In particular, we deduce that property (b) is not always divisorial and that there are period points of K3 type corresponding to GM fourfolds without a Hodge-associated K3 surface.

We also prove its natural extension to the twisted case, as in \cite{Huy} for cubic fourfolds.

\begin{thm}
\label{thm4}
Let $X$ be a Hodge-special GM fourfold with discriminant $d$ such that $Y_A^{\geq 3} = \emptyset$. Then $\tilde{Y}_A$ is birational to a moduli space of stable twisted sheaves on a K3 surface $S$ if and only if $d$ satisfies $(\ast\ast')$.
\end{thm}

Finally, we determine the numerical condition on the discriminant $d$ of a Hodge-special GM fourfold in order to have $\tilde{Y}_A$ birational to the Hilbert scheme $S^{[2]}$ on a K3 surface $S$; this condition is stricter than that of $(\ast\ast)$, as proved in \cite{Add} for cubic fourfolds (see Remark \ref{d=50}). 
\begin{thm}
\label{thm2}
Let $X$ be a Hodge-special GM fourfold of discriminant $d$ such that $Y_A^{\geq 3}=\emptyset$. Then $\tilde{Y}_A$ is birational to the Hilbert square $S^{[2]}$ of a K3 surface $S$ if and only if $d$ satisfies the condition
\begin{equation*}\tag{$\ast$$\ast$$\ast$}
\label{ast3}
a^2d=2n^2+2 \quad \text{for }a,n \in \Z.
\end{equation*}
\end{thm}

The strategy to prove these results relies on the definition of the Mukai lattice for the \emph{Kuznetsov component} (or the K3 category), which is a non commutative K3 surface arising from the semiorthogonal decomposition of the derived category of a GM variety constructed in \cite{KP} (see Section 2.4). The Mukai lattice is defined as done in \cite{AT} by Addington and Thomas for cubic fourfolds; actually, we can prove the analogue of their results, using the vanishing lattice of a GM fourfold instead of the primitive degree-four lattice of cubic fourfolds. In particular, following the work of Addington, this allows us to apply Propositions 4 and 5 of \cite{Add} and, then, to prove Theorems \ref{thm1} and  \ref{thm2}. On the other hand, we obtain that if a general GM fourfold has a homological associated K3 surface, then there is a Hodge-theoretic associated K3 surface (see Theorem \ref{thmHomHodge} for a more precise statement). 

\begin{rw}
In \cite[Proposition 2.1]{IM}, Iliev and Madonna prove that if a smooth double EPW sextic is birational to the Hilbert square $S^{[2]}$ on a K3 surface $S$ with polarization of degree-$d$, then the negative Pell equation $\mathcal{P}_{d/2}(-1): n^2-\frac{d}{2}a^2=-1$ is solvable. Thus Theorem \ref{thm2} is consistent with this necessary condition (see also Remark \ref{linkIM}).

Finally, in \cite[Corollary 7.6]{DM}, Debarre and Macrì prove that the Hilbert square of a general polarized K3 surface of degree-$d$ is isomorphic to a double EPW sextic if and only if the Pell equation $\mathcal{P}_{d/2}(-1)$ is solvable and the equation $\mathcal{P}_{2d}(5): n^2-(2d)a^2=5$ is not. By Theorem \ref{thm2}, we see that the birationality to this Hilbert scheme is obtained by relaxing the second condition on $\mathcal{P}_{2d}(5)$.
\end{rw} 

\begin{plan}
In Section 2 we recall some preliminary facts about Hodge-special GM fourfolds and their Kuznetsov component. In Section 3.1 we define the Mukai lattice and we reinterpret the definition of Hodge-special GM fourfolds with a certain discriminant via their Mukai lattice. In Section 3.2 we characterize the condition of having a Hodge-associated K3 surface by the existence of a primitively embedded hyperbolic lattice in the algebraic part of the Mukai lattice, for general Hodge-special GM fourfolds and for non general GM fourfolds satisfying a precise condition. Then, we prove Theorem \ref{thm3intro}. In Section 3.3 we discuss the construction of GM fourfolds which do not realize the equivalence of Theorem \ref{thm1}. Section 4.1 is devoted to the proofs of Theorem \ref{thm1} and Theorem \ref{thm4}. In Section 4.2 we prove Theorem \ref{thm2}. 
\end{plan}

\begin{ack}
I would like to thank Nicolas Addington for his useful comments. I am grateful to Chiara Camere, Alice Garbagnati, Alexander Kuznetsov, Simone Panozzo and Alex Perry for interesting suggestions. 
It is a pleasure to thank Emanuele Macrì for inspiring conversations, and my advisor Paolo Stellari for helping me throughout the preparation of this work. I thank the referees for expert suggestions and for careful reading of this paper. 

The author would like to acknowledge the National Group of Algebraic and Geometric Structures and their Applications (GNSAGA-INdAM) for financial support, and Institut Henri Poincaré, where this paper was completed. 
\end{ack}

\section{Background on Gushel-Mukai fourfolds}

The aim of this section is to recall some definitions and properties concerning Hodge-special GM fourfolds and to fix the notation. Our main references are \cite{DIM}, \cite{DK1} \cite{DK2} and \cite{KP}. 

\subsection{Geometry of GM fourfolds}
Let $V_5$ be a $5$-dimensional complex vector space; we denote by $\G(2,V_5)$ the Grassmannian of $2$-dimensional subspaces of $V_5$, viewed in $\p(\bigwedge^2 V_5) \cong \p^9$ via the Pl\"ucker embeddig. Let $\CG(2,V_5) \subset \p(\C \oplus \bigwedge^2 V_5) \cong \p^{10}$ be the cone over $\G(2,V_5)$ of vertex $\nu:=\p(\C)$. 

\begin{dfn}
A \textbf{GM fourfold} is a smooth intersection of dimension four
$$X=\CG(2,V_5) \cap \p(W) \cap Q,$$
where $W$ is a $9$-dimensional vector subspace of $\C \oplus \bigwedge^2 V_5$ and $Q$ is a quadric hypersurface in $\p(W)\cong \p^{8}$. 
\end{dfn}

\noindent Notice that $\nu$ does not belong to $X$, because $X$ is smooth. Thus, the linear projection from $\nu$ defines a regular map
$$\gamma_X: X \rightarrow \G(2,V_5)$$
called the \textbf{Gushel map}. We denote by $\mathcal{U}_X$ the pullback via $\gamma_X$ of the tautological rank-$2$ subbundle of $\G(2,V_5)$. We can distinguish two cases:
\begin{itemize}
\item If the linear space $\p(W)$ does not contain $\nu$, then $X$ is isomorphic to a quadric section of the linear section of the Grassmannian $\G(2,V_5)$ given by the projection of $\p(W)$ to $\p(\bigwedge^2 V_5)$. GM fourfolds of this form are \textbf{ordinary}.
\item If $\nu$ is in $\p(W)$, the linear space $\p(W)$ is a cone over $\p(W/\C) \cong \p^7 \subset \p(\bigwedge^2 V_5)$. Then, $X$ is a double cover of the linear section $\p(W/\C) \cap \G(2,V_5)$, branched along its intersection with a quadric. In this case, we say that $X$ is \textbf{special}. 
\end{itemize} 

We denote by $\sigma_{i,j} \in H^{2(i+j)}(\G(2,V_5),\Z)$ the Schubert cycles on $\G(2,V_5)$ for every $3 \geq i \geq j \geq 0$ and we set $\sigma_i:=\sigma_{i,0}$. The restriction $h:=\gamma_X^*\sigma_1$ of the hyperplane class on $\p(\C \oplus \bigwedge^2 V_5)$ defines a natural polarization of degree-$10$ on $X$. An easy computation using adjunction formula shows that $X$ is a Fano fourfold with canonical class $-2h$. 

The moduli stack $\mathcal{M}_4$ of GM fourfolds is a Deligne-Mumford stack of finite type over $\C$, of dimension $24$ (see \cite[Proposition A.2]{KP}). 

\subsection{Period map and period points}
We recall the Hodge numbers of $X$: 
  \[
\begin{tabular}{ccccccccccccccc}
&&&&&&&1\\
&&&&&&0&&0&\\
&&&&&0&&1&&0\\
&&&&0&&0&&0&&0\\
&&&0&&1&&22&&1&&0
\end{tabular}
\]
(see \cite[Lemma 4.1]{IMani}). Notice that $H^{\bullet}(X,\Z)$ is torsion free by \cite[Proposition 3.4]{DK2}. The classes $h^2$ and $\gamma_X^*\sigma_{2}$ span the embedding of the rank-two lattice $H^4(\G(2,V_5),\Z)$ in $H^4(X,\Z)$. The \textbf{vanishing lattice} of $X$ is the sublattice
$$H^4(X,\Z)_{00}:=\left\lbrace x \in H^4(X,\Z): x \cdot \gamma_X^{*}(H^4(\G(2,V_5),\Z))=0 \right\rbrace.$$
By \cite[Proposition 5.1]{DIM}, we have an isomorphism of lattices
$$H^4(X,\Z)_{00} \cong E_8^2 \oplus U^2 \oplus I_{2,0}(2)=:\Lambda.$$  
Here $U$ is the hyperbolic plane $(\Z^2,\begin{pmatrix}
0 & 1 \\ 
1 & 0
\end{pmatrix})$, $E_8$ is the unique even unimodular lattice of signature $(8,0)$ and $I_{r,s}:=I_1^r \oplus I_1(-1)^s$, where $I_1$ is the lattice $\Z$ with bilinear form $(1)$ (given a lattice $L$ and a non-zero integer $m$, we denote by $L(m)$ the lattice $L$ with the intersection form multiplied by $m$). We point out that the lattice $I_{2,0}(2)$ is also denoted by $A_1^2$ in the literature.  

Let $e$ and $f$ be two classes in $I_{22,2}$ of square $2$ and $e \cdot f=0$, which generate the orthogonal of $\Lambda$ in $I_{22,2}$. The choice of an isometry $\phi: H^4(X,\Z)\cong I_{22,2}$ sending $\gamma_X^*\sigma_{1,1}$ and $\gamma_X^*(\sigma_2-\sigma_{1,1})$ to $e$ and $f$ respectively, and such that $\phi(H^4(X,\Z)_{00}) = \Lambda$, determines a \textbf{marking} for $X$. Notice that the Hodge structure on the vanishing lattice is of K3 type. Let $\text{\~{O}}(\Lambda)$ be the subgroup of automorphisms of $\text{O}(\Lambda)$ acting trivially on the discriminant group $d(\Lambda)$. The groups $\text{\~{O}}(\Lambda)$ and $\text{O}(\Lambda)$ act properly discontinuously on the complex variety
\begin{equation}
\label{locperiodom}
\Omega:= \lbrace w \in \p(\Lambda \otimes \C): w \cdot w =0, w \cdot \bar{w}<0 \rbrace.
\end{equation}
The \textbf{global period domain} is the quotient $\mathcal{D}:=\text{\~{O}}(\Lambda) \setminus \Omega$, which is an irreducible quasi-projective variety of dimension $20$. We observe that two markings differ by the action of an element in $\text{\~{O}}(\Lambda)$. It follows that the \textbf{period map} $p: \mathcal{M}_4 \rightarrow \mathcal{D}$, which sends $X$ to the class of the one-dimensional subspace $H^{3,1}(X)$, is well-defined. As a map of stacks, $p$ is dominant with $4$-dimensional smooth fibers (see \cite[Theorem 4.4]{DIM}). The \textbf{period point} of $X$ is the image $p(X)$ in $\mathcal{D}$.

As proved in \cite{DK2}, the period point of a general GM fourfold is determined by that of its associated double EPW sextic. More precisely, let $(V_6,V_5,A)$ be the Lagrangian data of $X$, as defined in the introduction. By the work of O'Grady, we can consider the closed subschemes 
\[
Y_A^{\geq l}:=\lbrace [U_1] \in \p(V_6): \text{dim}(A \cap (U_1 \wedge \bigwedge\nolimits^2 V_6)) \geq l \rbrace \quad \text{for } l \geq 0.
\]
Since $A$ has no decomposable vectors, we have $Y_A:=Y_A^{\geq 1}$ is a normal sextic hypersurface, called EPW sextic, which is singular along the integral surface $Y_A^{\geq 2}$. Moreover, $Y_A^{\geq 3}$ is finite and it is the singular locus of $Y_A^{\geq 2}$, while $Y_A^{\geq 4}$ is empty (see \cite[Proposition B.2]{DK1}). Let $\tilde{Y}_A$ be the double cover of the EPW sextic $Y_A$ branched over $Y_A^{\geq 2}$. If $Y_A^{\geq 3}$ is empty, (e.g.\ for generic $A$), then the \textbf{double EPW sextic} $\tilde{Y}_A$ is a smooth hyperk\"ahler fourfold, deformation equivalent to the Hilbert scheme of two points on a K3 surface (see \cite[Theorem 1.1]{OG2}). In this case, the period point of $\tilde{Y}_A$ coincides with $p(X)$, as explained in the following result.   
 
\begin{thm}\emph{\cite[Theorem 5.1]{DK2}}
\label{periodEPW}
Let $X$ be a GM fourfold with associated Lagrangian data $(V_6,V_5,A)$. Assume that the double EPW sextic $\tilde{Y}_A$ is smooth (i.e. $Y_A^{\geq 3} = \emptyset$). Then, there is an isometry of Hodge structures 
$$H^4(X,\Z)_{00} \cong H^2(\tilde{Y}_A,\Z)_0(-1),$$ 
where $H^2(\tilde{Y}_A,\Z)_0$ is the degree-two primitive cohomology of $\tilde{Y}_A$ equipped with the Beauville-Bogomolov-Fujiki form $q$.
\end{thm}

\subsection{Hodge-special GM fourfolds}

\begin{dfn}
\label{defHodge-spe}
A GM fourfold $X$ is \textbf{Hodge-special} if  $H^{2,2}(X) \cap H^4(X,\Q)_{00} \neq 0$.
\end{dfn}

\noindent Equivalently, $X$ is Hodge-special if and only if $H^{2,2}(X,\Z)$ contains a rank-three primitive sublattice $K$ containing $\gamma_X^*(H^4(\G(2,V_5),\Z))$. Such a lattice $K$ is a \textbf{labelling} for $X$ and the discriminant of the labelling is the determinant of the intersection matrix on $K$. We say that $X$ \emph{has discriminant $d$} if it has a labelling of discriminant $d$. 

Notice that $d$ is positive and $d \equiv 0,2$ or $4(\text{mod\,8})$ (see \cite[Lemma 6.1]{DIM}). More precisely, the period point of a Hodge-special GM fourfold with discriminant $d$ belongs to an irreducible divisor $\mathcal{D}_d$ in $\mathcal{D}$ if $d \equiv 0 \pmod 4$, or to the union of two irreducible divisors $\mathcal{D}_d'$ and $\mathcal{D}_d''$ in $\mathcal{D}$ if $d \equiv 2 \pmod 8$ (see \cite[Corollary 6.3]{DIM}). The hypersurfaces $\mathcal{D}_d'$ and $\mathcal{D}_d''$ are interchanged by the involution $r_{\mathcal{D}}$, defined on $\mathcal{D}$ by exchanging $e$ and $f$.

Let $X$ be a Hodge-special GM fourfold with a labelling $K$ of discriminant $d$. The orthogonal $K^{\perp}$ of $K$ in $I_{22,2}$ is the \textbf{non-special lattice} of $X$; it is equipped with a Hodge structure induced by the Hodge structure on $H^4(X,\Z)$. A pseudo-polarized K3 surface $S$ of degree-$d$ is \textbf{Hodge-associated} to $(X,K)$ if the non-special cohomology $K^{\perp}$ is Hodge-isometric to the primitive cohomology lattice $H^2(S,\Z(-1))_0$. As proved in \cite[Proposition 6.5]{DIM}, this is equivalent to have $d$ satisfying $(\ast\ast)$. Moreover, if $p(X)$ is not in $\mathcal{D}_8$, then the pseudo-polarization is a polarization (see \cite[Proposition 6.5]{DIM}).  

\subsection{Kuznetsov component and K-theory}

The analogy between GM fourfolds and cubic fourfold is also reflected on their derived categories. Indeed, we denote by $\D(X)$ the bounded derived category of coherent sheaves on a GM fourfold $X$. By \cite[Proposition 4.2]{KP}, there exists a semiorthogonal decomposition of the form
$$\D(X)= \langle \A_X, \mathcal{O}_X, \mathcal{U}_X^*, \mathcal{O}_X(1),\mathcal{U}_X^*(1) \rangle,$$
where $\mathcal{A}_X$ is the right orthogonal to the subcategory generated by the exceptional objects
\begin{equation}
\label{excol}
\mathcal{O}_X,\mathcal{U}_X^*,\mathcal{O}_X(1),\mathcal{U}_X^*(1),
\end{equation}
in $\D(X)$. We say that $\A_X$ is the \textbf{Kuznetsov component} (or the \textbf{K3 category}) associated to $X$. The Kuznetsov component has the same Serre functor and the same Hochschild homology as the derived category of a K3 surface (see \cite[Proposition 4.5 and Lemma 5.3]{KP}). In particular, the Kuznetsov component can be viewed as a non commutative K3 surface. Moreover, if $X$ is an ordinary GM fourfold containing a quintic del Pezzo surface, then there exists a K3 surface $S$ and an equivalence $\A_X \xrightarrow{\sim} \D(S)$ (see \cite[Theorem 1.2]{KP}). 

We denote by $K_0(\A_X)$ the Grothendieck group of $\A_X$ and let $\chi$ be the Euler pairing. The numerical Grothendieck group of $\A_X$ is given by the quotient $K_0(\A_X)_{\text{num}}:=K_0(\A_X)/\ker\chi$. By the additivity with respect to semiorthogonal decompositions, we have the orthogonal direct sum 
$$K_0(X)_{\text{num}}=K_0(\A_X)_{\text{num}} \oplus \langle [\mathcal{O}_X],[\mathcal{U}_X^*],[\mathcal{O}_X(1)],[\mathcal{U}_X^*(1)] \rangle_{\text{num}} \cong K_0(\A_X)_{\text{num}} \oplus \Z^{4}$$ 
with respect to $\chi$. In particular, since the Hodge conjecture holds for $X$ (see \cite{CM}), it follows that
$$\text{rank}(K_0(\A_X)_{\text{num}})=\sum_k h^{k,k}(X,\Q) - 4=4+h^{2,2}(X,\Q)-4= h^{2,2}(X,\Q).$$ 
We recall the following lemma, which will be useful to study the relation between the Mukai lattice of $\A_X$ and the vanishing cohomology of $X$.
\begin{lemma}\emph{\cite[Lemma 5.14 and Lemma 5.16]{KP}}
\label{lemmabase}
If $X$ is a non Hodge-special GM fourfold, then $K_0(\A_X)_{\emph{num}} \cong \Z^2$ and it admits a basis such that the Euler form with respect to this basis is given by
\[
\begin{pmatrix}
 -2 & 0 \\ 
 0 & -2
 \end{pmatrix}.
\]  
\end{lemma}

We end this section with the explicit computation of the basis of Lemma \ref{lemmabase}. The Todd class of a GM fourfold $X$ is
$$\td(X)=1+h+\left( \frac{2}{3}h^2-\frac{1}{12}\gamma_X^*\sigma_2\right) +\frac{17}{60}h^3 + \frac{1}{10}h^4.$$
Let $P$ be a point in $X$, $L$ be a line lying on $X$, $\Sigma$ be the zero locus of a regular section of $\mathcal{U}_X^{\ast}$, $S$ be the complete intersection of two hyperplanes in $X$ and $H$ be a hyperplane section of $X$. Since $X$ is not Hodge-special, the structure sheaves of these subvarieties give a basis for the numerical Grothendieck group. Thus, an element $\kappa$ in $K_0(X)_{\text{num}}$ can be written as
$$\kappa=a[\mathcal{O}_X]+ b[\mathcal{O}_H]+c[\mathcal{O}_S]+d[\mathcal{O}_{\Sigma}] +e[\mathcal{O}_L]+f[\mathcal{O}_P],$$
for $a,b,c,d,e,f \in \Z$. A computation using Riemann-Roch gives that $\kappa$ belongs to $K_0(\A_X)$ if and only if it is a linear combination of the following classes:
\begin{align}
\label{basis}
\lambda_1 &:=2[\mathcal{O}_X]-[\mathcal{O}_{\Sigma}]-2[\mathcal{O}_L]+[\mathcal{O}_P] \\
\lambda_2 &:=4[\mathcal{O}_X]-2[\mathcal{O}_H]-[\mathcal{O}_S]-5[\mathcal{O}_L]+5[\mathcal{O}_P].\notag
\end{align}
It is easy to verify that the matrix they define with respect to the Euler form is as in Lemma \ref{lemmabase}. 

\begin{rmk}
Let $C$ be a generic conic in a GM fourfold $X$; we denote by $\mathcal{O}_C$ its structure sheaf. Notice that $\lambda_1$ is the class in the K-theory of $X$ of the projection of $\mathcal{O}_C(1)$ in $\A_X$. Indeed, the projection $\text{pr}: \D(X) \to \A_X$ is given by the composition $\text{pr}:=\mathbb{L}_{\mathcal{O}_X}\mathbb{L}_{\mathcal{U}_X^*}\mathbb{L}_{\mathcal{O}_X(1)}\mathbb{L}_{\mathcal{U}^*_X(1)}$ of the left mutation functors with respect to the exceptional objects (see for example \cite{BLMS}, Section 3 for the definition of left mutation). Performing this computation, we get that
$$[\text{pr}(\mathcal{O}_C(1))]=[\mathcal{O}_C(1)]-[\mathcal{O}_X(1)]+[\mathcal{U}_X^*]+[\mathcal{O}_X],$$ 
which has the same Chern character of $\lambda_1$. The second element $\lambda_2$ should be the class of an object in $\A_X$ obtained as the image of $\text{pr}(\mathcal{O}_C(1))$ via an autoequivalence of $\A_X$.
\end{rmk} 

\section{Mukai lattice for the Kuznetsov component}

In this section we describe the Mukai lattice of the Kuznetsov component. The main results of Section 3.1 are Proposition \ref{mukaivsvan}, where we prove that the vanishing lattice is Hodge isometric to the orthogonal of the lattice generated by $\lambda_1$ and $\lambda_2$ in the Mukai lattice, and Corollary \ref{correticoli}, where we determine Hodge-special GM fourfolds by their Mukai lattice. In Section 3.2 we relate the condition of having an associated K3 surface with the Mukai lattice (Theorem \ref{thmK3U}); as a consequence, we get Theorem \ref{thmHomHodge}, where we prove that the existence of a homological associated K3 surface implies that there is a Hodge-theoretic associated K3 surface for general Hodge-special GM fourfolds. Then we prove Theorem \ref{thm3intro}. We follow the methods introduced in \cite{AT} and \cite{Huy} for cubic fourfolds.   

\subsection{Mukai lattice and vanishing lattice}

Let $X$ be a GM fourfold. We denote by $K(X)_{\text{top}}$ the topological K-theory of $X$ which is endowed with the Euler pairing $\chi$. As recalled in Section 2.2, the group $H^{\bullet}(X,\Z)$ is torsion-free; by \cite[Section 2.5]{AH} (see also \cite{AT}, Theorem 2.1), it follows that also $K(X)_{\text{top}}$ is torsion-free. 

Inspired by \cite{AT}, we define the \textbf{Mukai lattice} of the Kuznetsov component $\A_X$ as the abelian group
$$\KT:=\lbrace \kappa \in K(X)_{\text{top}}: \chi([\mathcal{O}_X(i)],\kappa)=\chi([\mathcal{U}_X^*(i)],\kappa)=0 \, \text{ for }i=0,1 \rbrace$$
with the Euler form $\chi$. We point out that $\KT$ is torsion-free, because $K(X)_{\text{top}}$ is. We recall that the Mukai vector of an element $\kappa$ of $K(X)_{\text{top}}$ is given by
$$v(\kappa)=\ch(\kappa).\sqrt{\td(X)}$$
and it induces an isomorphism of $\Q$-vector spaces $v: K(X)_{\text{top}}\otimes \Q \cong H^{\bullet}(X,\Q)$. We define the weight-zero Hodge structure on the Mukai lattice given by pulling back via the isomorphism 
$$\KT \otimes \C \rightarrow \bigoplus_{p=0}^4 H^{2p}(X,\C)(p)$$
induced by $v$. It is also convenient to consider the Mukai lattice $\KT(-1)$ with weight-two Hodge structure $\bigoplus_{p+q=2}\tilde{H}^{p,q}(\A_X)$ and Euler form with reversed sign. In the following, we will use both conventions according to the situation. The \emph{Néron-Severi lattice} of $\A_X$ is 
$$N(\A_X)= \tilde{H}^{1,1}(\A_X,\Z):= \tilde{H}^{1,1}(\A_X)\cap \KT$$
and the \emph{trascendental lattice} $T(\A_X)$ is the orthogonal complement of the Néron-Severi lattice with respect to $\chi$. 

We observe that by \cite[Theorem 1.2]{KP}, there exist GM fourfolds $X$ such that the associated Kuznetsov component $\A_X$ is equivalent to the derived category of a K3 surface $S$. Moreover, any equivalence $\A_X \xrightarrow{\sim} \D(S)$ induces an isometry of Hodge structures $\KT(-1) \cong K(S)_{\text{top}}$, by the same argument used in \cite[Section 2.3]{AT}. We set $\tilde{\Lambda}:=U^4 \oplus E_8(-1)^2$ and we recall that $K(S)_{\text{top}}$ is isomorphic as a lattice to $\tilde{\Lambda}$. Since the definition of $\KT$ does not depend on $X$ (any two GM fourfolds are deformation equivalent), we deduce that the Euler form is symmetric on $\KT$ and $\KT$ is isomorphic as a lattice to $\tilde{\Lambda}(-1)= U^4\oplus E_8^2$.

We denote by $\langle \lambda_1,\lambda_2 \rangle^{\perp}$ the orthogonal complement with respect to the Euler pairing of the sublattice of $\KT$ generated by the objects $\lambda_1, \lambda_2$ determined in \eqref{basis}. In the next result, we explain the relation of this lattice with the vanishing lattice $H^4(X,\Z)_{00}$.

\begin{prop}
\label{mukaivsvan}
Let $X$ be a GM fourfold. Then the Mukai vector $v$ induces an isometry of Hodge structures
$$\langle \lambda_1,\lambda_2 \rangle^{\perp} \cong H^4(X,\Z)_{00}(2)=\langle h^2,\gamma_X^*\sigma_2 \rangle^{\perp}.$$
Moreover, for every set of $n$ objects $\zeta_1,\dots,\zeta_n$ in $K(\A_X)_{\emph{top}}$, the Mukai vector induces the isometry
$$\langle \lambda_1,\lambda_2, \zeta_1,\dots,\zeta_n \rangle^{\perp} \cong \langle h^2,\gamma_X^*\sigma_2,c_2(\zeta_1),\dots,c_2(\zeta_n)\rangle^{\perp}.$$
\end{prop}
\begin{proof}
By definition $\kappa$ belongs to $\langle \lambda_1,\lambda_2 \rangle^{\perp} \subset \KT$ if and only if
\begin{equation}
\begin{cases}
\label{sist2}
\chi(\mathcal{O}_X,\kappa)=\chi(\mathcal{O}_X(1),\kappa)=\chi(\mathcal{U}^*_X,\kappa)=\chi(\mathcal{U}^*_X(1),\kappa)=0\\
\chi(\lambda_1,\kappa)=\chi(\lambda_2,\kappa)=0.
\end{cases}
\end{equation}
The Chern character of $\kappa$ has the form
$$\ch(\kappa)=k_0+k_2h+k_4+k_6h^3+k_8h^4 \quad \text{for }k_0,k_2,k_6,k_8 \in \Q \text{ and } k_4 \in H^4(X,\Q).$$ 
Thus, using Riemann-Roch, we can express the conditions \eqref{sist2} as a linear system in the variables $k_0,k_2,k_4 \cdot h^2,k_4 \cdot \gamma_X^*\sigma_{1,1},k_6,k_8$. Since the equations are linearly independent, we obtain that the system \eqref{sist2} has a unique solution, i.e.
$$k_0=k_2=k_6=k_8=0 \quad \text{and} \quad k_4 \cdot h^2=k_4 \cdot \gamma_X^*\sigma_{1,1}=0.$$
In particular, $\ch(\kappa)$ belongs to $\langle 1,h,h^2, \gamma_X^*\sigma_2,h^3,h^4 \rangle^{\perp}=\langle h^2,\gamma_X^*\sigma_2\rangle^{\perp}$ in $H^4(X,\Q)$. Since $k_4 \cdot h=0$,  $v(\kappa)=k_4$, i.e.\ $v(\kappa)$ is in the sublattice $\langle h^2,\gamma_X^*\sigma_2\rangle^{\perp}$ of $H^4(X,\Q)$. Since the lowest-degree term of the Mukai vector is integral (see \cite[Section 2.5]{AH} and \cite[Proposition 3.4]{DK2}), we conclude that $\kappa$ belongs to $\langle \lambda_1,\lambda_2 \rangle^{\perp}$ if and only if $v(\kappa)$ is in $H^4(X,\Z)_{00}$. 

By \cite[Section 2.5]{AH}, we have $v: \langle \lambda_1,\lambda_2 \rangle^{\perp} \rightarrow H^4(X,\Z)_{00}(2)$ is injective. 
It remains to prove the surjectivity. It is possible to argue as in the proof of \cite[Proposition 2.3]{AT} using \cite[Section 2.5]{AH}. We propose an alternative way. We observe that the lattices $\langle \lambda_1,\lambda_2 \rangle^{\perp}$ and $H^4(X,\Z)_{00}$ have both rank-$22$. Notice that $\langle \lambda_1,\lambda_2 \rangle^{\perp}$ has signature $(20,2)$. Moreover, the discriminant group of $\langle \lambda_1,\lambda_2 \rangle^{\perp}$ is isomorphic to $(\mathbb{Z}/2\mathbb{Z})^2$, because the Mukai lattice is unimodular. On the other hand, by Section 2.2 (see \cite[Proposition 5.1]{DIM}), we deduce that $H^4(X,\Z)_{00}$ and $\langle \lambda_1,\lambda_2 \rangle^{\perp}$ have the same signature and isomorphic discriminant groups. Since the genus of such a lattice contains only one element by \cite[Theorem 1.14.2]{Ni}, we conclude that $v$ is an isometry which preserves the Hodge structures, as we wanted.

For the second part of the proposition, let $v(\zeta_i)=z_0+z_2h+z_4+z_6h^3+z_8h^4$ with $z_0,z_2,z_6,z_8 \in \Q$ and $z_4 \in H^4(X,\Q)$. Using the previous computation, we have 
$$0=\chi(\zeta_i,\kappa)=\int_X \exp\left(h \right)v(\zeta_i)^*\cdot k_4= k_4 \cdot z_4$$
for every $\kappa$ in $\langle \lambda_1,\lambda_2, \zeta_1,\dots,\zeta_n \rangle^{\perp}$. Since $z_4$ is by definition a linear combination of $c_2(\zeta_i)$, $h^2$ and $\gamma_X^*\sigma_2$, using again that $k_4$ is in $H^4(X,\Z)_{00}$, we deduce that $k_4 \cdot z_4=0$ if and only if $k_4 \cdot c_2(\zeta_i)=0$. This completes the proof of the statement. 
\end{proof}

We point out that the lattice $\langle \lambda_1,\lambda_2 \rangle$ has a primitive embedding in $\KT$ by \cite[Corollary 1.12.3]{Ni}. By Proposition \ref{mukaivsvan}, we have the isomorphism of lattices
$$\langle \lambda_1,\lambda_2 \rangle^{\perp} \cong H^4(X,\Z)_{00} \cong E_8^2 \oplus U^2 \oplus I_{2,0}(2).$$
On the other hand, the lattice $\langle \lambda_1,\lambda_2 \rangle$ is isomorphic to $I_{0,2}(2)$. Notice that by \cite[Theorem 1.14.4]{Ni}, there exists a unique (up to isomorphism) primitive embedding
$$i: I_{0,2}(2) \hookrightarrow \tilde{\Lambda}(-1)=E_8^2 \oplus U^4.$$
Let us denote by $f_1,f_2$ the standard generators of $I_{0,2}(2)$ and by $u_1,v_1$ (resp.\ $u_2,v_2$) the standard basis of the first (resp.\ the second) hyperbolic plane $U$. Then, we define $i$ setting
$$i(f_1)=u_1-v_1 \quad i(f_2)=u_2-v_2.$$
The orthogonal complement of $I_{0,2}(2)$ via $i$ is
$$I_{0,2}(2)^{\perp}\cong E_8^2 \oplus U^2 \oplus I_{2,0}(2).$$
In particular, we have an isometry $\phi: \KT \cong \tilde{\Lambda}(-1)$ such that
\begin{equation}
\label{isophi}
\phi(\lambda_1)=i(f_1), \quad \phi(\lambda_2)=i(f_2), \quad \phi(\langle \lambda_1,\lambda_2 \rangle^{\perp})\cong I_{0,2}(2)^{\perp}\cong E_8^2 \oplus U^2 \oplus I_{2,0}(2),
\end{equation}
which is equivalent to the data of a marking for $X$. Hence, we can write $p(X)=[\phi_{\C}(\tilde{H}^{2,0}(\A_X))]$.

Now, we prove that the isomorphism of Proposition \ref{mukaivsvan} extends to the quotients $\KT/\langle \lambda_1,\lambda_2 \rangle$ and $H^4(X,\Z)/\langle h^2,\gamma_X^*\sigma_2 \rangle$. The proof is analogous to that of \cite[Proposition 2.4]{AT}.

\begin{prop}
\label{c2quoz}
The second Chern class induces a group isomorphism 
$$\bar{c}_2: \frac{K(\A_X)_{\emph{top}}}{\langle \lambda_1,\lambda_2 \rangle} \rightarrow \frac{H^4(X,\Z)}{\langle h^2,\gamma_X^*\sigma_2 \rangle}.$$
\end{prop}
\begin{proof}
The composition of the projection $p: H^4(X,\Z) \twoheadrightarrow H^4(X,\Z)/\langle h^2,\gamma_X^*\sigma_2 \rangle$ with $c_2$ is a group homomorphism, because 
$$c_2(\kappa_1 + \kappa_2)=c_2(\kappa_1)+c_1(\kappa_1)c_1(\kappa_2)+c_2(\kappa_2)=c_2(\kappa_1)+mh^2+c_2(\kappa_2) \quad \text{for }m \in \Z.$$
Since the second Chern classes of $\lambda_1$ and $\lambda_2$ are respectively 
$$c_2(\lambda_1)=2 h^2 \quad \text{and} \quad c_2(\lambda_2)=-\gamma_X^*\sigma_{1,1},$$
it follows that $\langle \lambda_1,\lambda_2 \rangle$ is in the kernel of $p \circ c_2$. In particular, the induced morphism $\bar{c}_2$ of the statement is well-defined. 

Notice that $\bar{c}_2$ is injective. Indeed, let $\kappa$ be an element in $\KT$ such that $c_2(\kappa)$ belongs to the sublattice $\langle h^2,\gamma_X^*\sigma_2 \rangle$. In particular, $\kappa$ is an element of $K(X)_{\text{top}}$ such that $\ch(\kappa)$ belongs to $H^0(X,\Z) \oplus H^2(X,\Z) \oplus \Z\langle h^2,\gamma_X^*\sigma_2 \rangle \oplus H^6(X,\Z) \oplus H^8(X,\Z)$. Then $\kappa$ is a linear combination of $[\mathcal{O}_X], [\mathcal{O}_H], [\mathcal{O}_S], [\mathcal{O}_{\Sigma}],[\mathcal{O}_L],[\mathcal{O}_P]$ with the notation of Section 2.4, because $X$ is AK-compatible (see \cite[Section 5]{KP}). Since it belongs to $\KT$, by the same computation done in the end of Section 2.4, we deduce that $\kappa$ is a linear combination of $\lambda_1$ and $\lambda_2$, as we claimed.

Finally, we show that $\bar{c}_2$ is surjective. Let $T$ be a class in $H^4(X,\Z)$. By \cite[Theorem 2.1(3)]{AT}, there exists $\tau$ in $K(X)_{\text{top}}$ such that $v(\tau)$ is the sum of $-T$ with highter degree terms. Then the projection $\text{pr}(\tau)$ of $\tau$ in $\KT$ is a linear combination of $\tau$ and the classes of the exceptional objects in \eqref{excol}. Since the Chern classes of the exceptional objects are all multiples of $h^i$ and $\gamma_X^*\sigma_{1,1}$, it follows that $c_2(\text{pr}(\tau))$ differs form $c_2(\tau)$ by a linear combination of $h^2$ and $\gamma_X^*\sigma_{1,1}$. We conclude that $\bar{c_2}(\text{pr}(\tau))=c_2(\tau)=T$ in $H^4(X,\Z)/\langle h^2,\gamma_X^*\sigma_2 \rangle$. 
\end{proof}

\begin{rmk}
Notice that the image of the algebraic K-theory $K(\A_X)$ in $\KT$ is contained in $N(\A_X)$. However, we do not know if the opposite inclusion holds, because it is not clear if every Hodge class in $H^{2,2}(X,\Z)$ comes from an algebraic cycle with integral coefficients. In the case of cubic fourfolds the integral Hodge conjecture holds by the work of Voisin (see \cite{Voi}); thus, in \cite[Proposition 2.4]{AT} they use this fact to prove that the $(1,1)$ part of the Hodge structure on the Mukai lattice is identified with $K(\A_X)_{\text{num}}$. 

Voisin's argument should work also for GM fourfolds, but it requires to give a description of the intermediate Jacobian of a GM threefold, as done in \cite[Theorem 5.6]{MT} and \cite[Theorem 1.4]{Dru} in the cubic threefolds case. An other approach could be firstly to construct Bridgeland stability conditions for the Kuznetsov component (e.g.\ as in \cite{BLMS} for the Kuznetsov component of a cubic fourfold). Then, to deduce the integral Hodge conjecture by an argument on moduli spaces of stable objects with given Mukai vector, along the same lines as in \cite{BLMSpre} where they develop the argument for cubic fourfolds. 
\end{rmk} 

Finally, we need the following lemma, which is a consequence of Proposition \ref{c2quoz}; the proof is the same as that of \cite[Proposition 2.5]{AT}, so we skip it.
\begin{lemma}
\label{lemmauguale}
Let $\kappa_1, \dots, \kappa_n$ be in $K(\A_X)_{\emph{top}}$; we define the sublattices
$$M_K:=\langle \lambda_1,\lambda_2,\kappa_1,\dots,\kappa_n \rangle \subset K(\A_X)_{\emph{top}}$$
and 
$$M_H:=\langle h^2, \gamma_X^*\sigma_2,c_2(\kappa_1),\dots,c_2(\kappa_n) \rangle \subset H^4(X,\Z).$$
\begin{enumerate}
\item An element $\kappa$ of $K(\A_X)_{\emph{top}}$ is in $M_K$ if and only if $c_2(\kappa)$ is in $M_H$.
\item $M_H$ is primitive if and only if $M_K$ is.
\item $M_H$ is non degenerate if and only if $M_K$ is.
\item If $M_K$ is in $N(\A_X)$, then $M_K$ and $M_H$ are non-degenerate.
\item If $M_K$ and $M_H$ are non-degenerate, then $M_H$ has signature $(r,s)$ if and only if $M_K$ has signature $(r-2,s+2)$ and they have isomorphic discriminant groups.  
\end{enumerate}
\end{lemma}

\begin{cor}
\label{correticoli}
The period point of a Hodge-special GM fourfold $X$ belongs to the divisor $\mathcal{D}_d$ (resp.\ to the union of the divisors $\mathcal{D}_d'$ and $\mathcal{D}_d''$) for $d \equiv 0 \pmod 4$ (resp.\ for $d \equiv 2 \pmod 8$) if and only if there exists a primitive sublattice $M_K$ of $N(\A_X)$ of rank-$3$ and discriminant $d$ which contains $\langle \lambda_1,\lambda_2 \rangle$.
\end{cor}
\begin{proof}
As recalled in Section 2.3, the period point of $X$ satisfies the condition of the statement if and only if there is a labelling $M_H$ of $H^{2,2}(X,\Z)$ with discriminant $d$. The claim follows from Lemma \ref{lemmauguale}.
\end{proof}

\subsection{Associated (twisted) K3 surface and Mukai lattice}

The first result of this section characterizes period points of general Hodge-special GM fourfolds by their Mukai lattice. It is analogous to \cite[Theorem 3.1]{AT} for cubic fourfolds and the proof develops in a similar fashion. 

\begin{thm}
\label{thmK3U}
Let $X$ be a Hodge-special GM fourfold. If $X$ admits a Hodge-associated K3 surface, then $N(\A_X)$ contains a copy of the hyperbolic plane. Moreover, the converse holds assuming one of the following conditions:
\begin{enumerate}
\item $X$ is general (i.e.\ $H^{2,2}(X,\Z)$ has rank-$3$);
\item There is an element $\tau$ in the hyperbolic plane such that $\langle \lambda_1, \lambda_2, \tau \rangle$ has discriminant $d \equiv 2$ or $4 \pmod 8$. 
\end{enumerate}
\end{thm}
\begin{proof}
Assume that $X$ has a Hodge-associated K3 surface; as recalled in the introduction and in Section 2.3, there exists a labelling $M_H$ whose discriminant $d$ satisfies $(\ast\ast)$. Equivalently, by Corollary \ref{correticoli}, there exists a primitive sublattice $M_K$ in $N(\A_X)$ of rank-$3$ containing $\langle \lambda_1,\lambda_2 \rangle$, with same discriminant $d$. Thus, there exists a rank-one primitive sublattice $\Z w$ and a primitive embedding $j: \Z w \hookrightarrow U^3 \oplus E_8^2$ with $w^2=-d$, such that $M_K^{\perp}$ in $\KT$ is isomorphic to $\Z w^{\perp}$. Adding $U$ to both sides of $j$, we get the primitive embedding of $U \oplus \Z w$ in $\tilde{\Lambda}(-1)$. Since $U \oplus \Z w$  and $M_K \subset \KT \cong \tilde{\Lambda}(-1)$ have isomorphic orthogonal complements, they have isomorphic discriminant groups by \cite[Corollary 1.6.2]{Ni}. Since one contains $U$, they are isomorphic by \cite[Corollary 1.13.4]{Ni}. In particular, we conclude that $U$ is contained in $M_K \subset N(\A_X)$, as we wanted.

Conversely, let $X$ be as in the second part of the statement and let $\kappa_1,\kappa_2$ be two classes in $N(\A_X)$ spanning a copy of $U$. Notice that $\langle \lambda_1,\lambda_2 \rangle$ is negative definite and $U$ is indefinite; hence, the lattice $\langle \lambda_1,\lambda_2,\kappa_1,\kappa_2 \rangle$ has rank three or four. We distinguish these two cases.

\begin{rank3}
Let $M_K$ be the saturation of $\langle \lambda_1,\lambda_2,\kappa_1,\kappa_2 \rangle$ and we denote by $d$ its discriminant. We have the inclusions $U \subset M_K \subset \KT \cong \tilde{\Lambda}(-1)$. Since $U$ is unimodular, there exists a rank-one sublattice $\Z w$ with $w^2=-d$ such that $M_K \cong U \oplus \Z w$. On the other hand, the orthogonal to $U$ in $\KT$ is an even unimodular lattice of signature $(19,3)$; thus it is isomorphic to $U^3 \oplus E_8^2$. As a consequence, $M_K^{\perp}$ in $\KT$ is isomorphic to $\Z w^{\perp}$ in $U^3 \oplus E_8^2$. As observed before, this is equivalent to the existence of a labelling $M_H$ for $X$ of discriminant $d$ satisfying condition $(\ast\ast)$. This ends the proof in the rank-three case. In particular, this proves the statement for $X$ general.  
\end{rank3}

\begin{rank4}
Consider the rank-three lattices of the form $\langle \lambda_1,\lambda_2, x\kappa_1 +y\kappa_2 \rangle$, where $x$ and $y$ are integers not both zero. We define the quadratic form
$$Q(x,y):=\begin{cases}
\text{disc}(\langle \lambda_1,\lambda_2, x\kappa_1 +y\kappa_2 \rangle) & \text{ if } x \neq 0 \text{ or } y \neq 0 \\
0 & \text{ if } x=y=0.
\end{cases}$$
We observe that the second Chern class $c_2(x\kappa_1 +y\kappa_2)$ is in $H^{2,2}(X)$; hence, by the Hodge-Riemann bilinear relations and Lemma \ref{lemmauguale}, it follows that $Q(x,y)$ is positive unless $x=y=0$.

Let 
\[
\begin{pmatrix}
 -2 & 0 & k & m \\ 
 0 & -2 & l & n \\ 
 k & l & 0 & 1 \\ 
 m & n & 1 & 0
 \end{pmatrix} 
\]
be the matrix defined by the Euler pairing on the lattice $\langle \lambda_1,\lambda_2,\kappa_1,\kappa_2 \rangle$. We have 
\begin{align*}
Q(x,y)&= \begin{vmatrix}
-2 & 0 & kx+my \\ 
0 & -2 & lx+ny \\ 
kx+my & lx+ny & 2xy
\end{vmatrix} \\
&=8xy+2(kx+my)^2+2(lx+ny)^2 \\
&=(2k^2+2l^2)x^2+(8+4km+4ln)xy+(2m^2+2n^2)y^2.
\end{align*}
We set
$$A:=2k^2+2l^2, \quad B:=8+4km+4ln, \quad C:=2m^2+2n^2.$$
We denote by $h$ the greatest common divisor of $A, B$ and $C$; notice that $h$ is even. We set
$$a=A/h, \quad b=B/h, \quad c=C/h$$
and we have $Q(x,y)=hq(x,y)$, where
$$q(x,y)=ax^2+bxy+cy^2.$$
In the next lemmas we prove that $h$ satisfies $(\ast \ast)$ and that there exist integers $x$ and $y$ such that $q(x,y)$ represents a prime $p \equiv 1 \pmod 4$ . 

\begin{lemma}
\label{lemma1}
The only odd primes that divide the greatest common divisor $h$ of the coefficients of $Q$ are $\equiv 1 \pmod 4$. Moreover, we have $8 \nmid h$.
\end{lemma}
\begin{proof}
Let $\Z[\sqrt{-1}]$ be the domain of Gaussian integers with the Euclidean norm $|\:|$. We set
$$\alpha:= k+l\sqrt{-1} \quad \text{and} \quad \gamma:=m+n\sqrt{-1}.$$
We rewrite the coefficients of $Q$ as
$$A=2|\alpha|^2, \quad B=4\text{Re}(\alpha\bar{\gamma})+8, \quad C=2|\gamma|^2.$$
Suppose that $p$ is an odd prime which is not congruent to $1$ modulo $4$, i.e. $p \equiv 3 \pmod 4$. Then $p$ is prime in $\Z[\sqrt{-1}]$ (see \cite[Proposition 4.18]{Cox}). Thus if $p$ divides $A=2\alpha\bar{\alpha}$, then $p$ divides $\alpha$. In particular, $p$ divides $\text{Re}(\alpha\bar{\gamma})$; so $p$ does not divide $\text{Re}(\alpha\bar{\gamma})+2$. It follows that $p$ does not divide $B$ and we conclude that $p \nmid h$. 

For the second part, we observe that $8 \mid h$ if and only if $k, l, m, n$ are even. In this case, we have $8 \mid Q(x,y)$ for every $x, y \in \Z$. However, the assumption we made in item 2 of the theorem exclude this possibility. 
\end{proof}

\begin{lemma}
\label{lemma2}
We have $a \nequiv 3 \pmod 4$, $c \nequiv 3 \pmod 4$, and $b$ is even. 
\end{lemma}
\begin{proof}
By definition we have 
$$k^2+l^2=\frac{h}{2}a \quad \text{and} \quad m^2+n^2=\frac{h}{2}c.$$
Notice that if an odd prime $\equiv 3 \pmod 4$ divides the sum of two squares, then it has to appear with even exponent (see \cite[Corollary 5.14]{NZ}). Since by Lemma \ref{lemma1} the only odd primes dividing $h$ are $\equiv 1 \pmod 4$, a prime $\equiv 3 \pmod 4$ appears in the prime factorization of $a$ and $c$ only with even exponent. This gives the first part of the claim.

Now, we prove that $b$ is odd if and only if $8 \mid h$. This implies the desired statement by the second part of Lemma  \ref{lemma1}.

Assume that $b$ is odd. Since
$$B=4(2 + \text{Re}(\alpha \bar{\gamma}))=hb,$$
we have $4 \mid h$. Thus, $4$ divides $A=2|\alpha|^2$ and $C=2|\gamma|^2$. It follows that $(1 + \sqrt{-1}) \mid \alpha$ and $(1 + \sqrt{-1}) \mid \gamma$, which implies that $2 \mid \alpha \bar{\gamma}$. We conclude that $8 \mid B$ and thus $8 \mid h$.

Conversely, assume that $8 \mid h$; arguing as above, we see that $8 \mid B$. Notice that $2 \nmid h/8$, because otherwise $2 \mid B/8$, in contradiction with the fact that $B=8(1+2r)$. Since
$$\frac{B}{8}=1+2r=\frac{h}{8}b,$$
we conclude that $b$ is odd.
\end{proof}

\begin{lemma}
\label{lemma3}
There exist integers $x$ and $y$ such that $q(x,y)$ is a prime $p \equiv 1 \pmod 4$.
\end{lemma}
\begin{proof}
We adapt part of the proof of \cite[Proposition 3.3]{AT} to our case. Let us list all the possible forms $q(x,y)$ modulo $4$, using the restrictions given by Lemma \ref{lemma2}:
\begin{center}
\begin{tabular}{|c|c|}
\hline 
\multicolumn{2}{|c|}{For $b \equiv 0 \pmod 4$:}  \\
\hline 
$a \equiv 0 \pmod 4$ & $0, y^2, 2y^2$ \\ 
\hline 
$a \equiv 1 \pmod 4$ & $x^2, x^2+y^2$ \\ 
\hline 
$a \equiv 2 \pmod 4$ & $2x^2, 2x^2+2y^2$ \\ 
\hline  
\end{tabular} 
\end{center}

\begin{center}
\begin{tabular}{|c|c|}
\hline 
\multicolumn{2}{|c|}{For $b \equiv 2 \pmod 4$:}  \\
\hline 
$a \equiv 0 \pmod 4$ & $2xy, 2xy+2y^2$ \\ 
\hline 
$a \equiv 1 \pmod 4$ & $x^2+2xy+y^2, x^2+2xy+2y^2$ \\ 
\hline 
$a \equiv 2 \pmod 4$ & $2x^2+2xy, 2x^2+2xy+y^2, 2x^2+2xy+2y^2$ \\ 
\hline
\end{tabular} 
\end{center}
Notice that we have excluded the cases 
$$x^2+2y^2, \quad x^2+2xy, \quad 2x^2+y^2, \quad 2xy+y^2,$$
because by completing the square we get
$$x^2+2xy=(x+y)^2-y^2 \equiv (x+y)^2+3y^2 \pmod 4$$
and
$$x^2+2y^2=(x+y)^2-2xy+y^2 \equiv (x+y)^2+2xy+y^2 \equiv 2(x+y)^2+3x^2 \pmod 4,$$
which is not possible by Lemma \ref{lemma2}.

We exclude the cases corresponding to a non primitive form, i.e.\ 
$$0,2y^2, 2x^2, 2x^2+2y^2, 2xy, 2xy+2y^2, 2x^2+2xy, 2x^2+2xy+2y^2.$$ 

In the other cases, we find that $q$ can represent only numbers which are $\equiv 0$ or $1 \pmod 4$ (i.e.\ $y^2, x^2, (x+y)^2, x^2+2xy+2y^2, 2x^2+2xy+y^2$), or $\equiv 0, 1$ or $2 \pmod 4$ (i.e.\ $x^2+y^2$). Since a primitive positive definite form represents infinitely many primes, it must represent a prime $\equiv 1 \pmod 4$. 
\end{proof}
\end{rank4}
We observe that $h$ satisfies $(\ast\ast)$ by Lemma \ref{lemma1}. Thus, by Lemma \ref{lemma3} we conclude that there exist some integers $x$ and $y$ not both zero such that the discriminant of the lattice $\langle \lambda_1,\lambda_2, x\kappa_1 +y\kappa_2 \rangle$ satisfies $(\ast\ast)$. This observation implies the proof of the statement. Indeed, if $M_K$ is the saturation of this lattice, then its discriminant still satifies condition $(\ast\ast)$, because $\text{discr}(\langle \lambda_1,\lambda_2, x\kappa_1 +y\kappa_2 \rangle)=i^2 \text{discr}(M_K)$, and $M_K$ has rank-three. By the same argument used at the end of the rank-three case, we deduce that $X$ has a Hodge-associated K3 surface.
\end{proof}

In Section 3.3 we give examples of GM fourfolds having a primitively embedded hyperbolic plane in the algebraic part of the Mukai lattice, but without a Hodge-associated K3 surface.

A consequence of Theorem \ref{thmK3U} is that the condition of having a homological associated K3 surface implies the existence of a Hodge-associated K3 surface for general GM fourfolds. This is analogous to the easy implication of \cite[Theorem 1.1]{AT}.

\begin{thm}
\label{thmHomHodge}
Let $X$ be a GM fourfold such that $\A_X$ is equivalent to the derived category of a K3 surface $S$. Under the hypothesis of the second part of Theorem \ref{thmK3U}, the GM fourfold $X$ has discriminant $d$ with $d$ satisfying $(\ast\ast)$.
\end{thm}
\begin{proof}
Assume that there is an equivalence $\Phi: \A_X \xrightarrow{\sim} \D(S)$ where $S$ is a K3 surface. We observe that $K(S)_{\text{num}}$ contains a copy of the hyperbolic plane spanned by the classes of the structure sheaf of a point and the ideal sheaf of a point. Since $\Phi$ induces an isometry of Hodge structures $\KT(-1) \cong K(S)_{\text{top}}$, it follows that $U$ is contained in $N(\A_X)$. Applying Theorem \ref{thmK3U}, we deduce the proof of the result. 
\end{proof}

In the last part of this section we show that period points of Hodge-special GM fourfolds with an associated twisted K3 surface are organized in a countable union of divisors determined by the value of the discriminant. The argument essentially follows \cite[Section 2]{Huy}. To this end, given a GM fourfold $X$, we consider the Mukai lattice $\KT(-1)$ with the weight-two Hodge structure and Euler pairing with reversed sign. Accordingly, the local period domain is given by
$$\Omega:= \lbrace w \in \p(I_{2,0}(2)^{\perp} \otimes \C): w \cdot w =0, w \cdot \bar{w}>0 \rbrace$$
changing the sign of the definition in $\eqref{locperiodom}$ and identifying $\Lambda=I_{2,0}(2)^{\perp}$.  We set $\mathcal{Q}=\lbrace x \in \p(\tilde{\Lambda}\otimes \C): x^2=0, (x.\bar{x})>0 \rbrace$ and we consider the canonical embedding of $\Omega$ in $\mathcal{Q}$.   

We recall that a point $x$ of $\mathcal{Q}$ is of \emph{K3 type} (resp.\ \emph{twisted K3 type}) if there exists a K3 surface $S$ (resp.\ a twisted K3 surface $(S,\alpha)$) such that $\tilde{\Lambda}$ with the Hodge structure defined by $x$ is Hodge isometric to $\tilde{H}(S,\Z)$ (resp.\ $\tilde{H}(S,\alpha,\Z)$) (see \cite[Definition 2.5]{Huy}). By \cite[Lemma 2.6]{Huy}, a point $x \in \mathcal{Q}$ is of K3 type (resp.\ of twisted K3 type) if and only if there exists a primitive embedding of $U$ (resp.\ an embedding of $U(n)$) in the $(1,1)$-part of the Hodge structure defined by $x$ on $\tilde{\Lambda}$. We denote by $\mathcal{D}_{\text{K3}}$ (resp.\ $\mathcal{D}_{\text{K3}'}$) the set of points of $\Omega$ of K3 type (resp.\ of twisted K3 type). 

\begin{dfn}
\label{deftwist}
A GM fourfold $X$ has an associated twisted K3 surface if the period point $p(X)$ comes from a point in $\mathcal{D}_{\text{K3}'}$.
\end{dfn}

\begin{rmk}
Notice that if $X$ has a Hodge-associated K3 surface, then it corresponds to a point $x$ of K3 type. In fact, it follows from the first part of Theorem \ref{thmK3U} and \cite[Lemma 2.6]{Huy}. Moreover, the converse holds for general Hodge-special GM fourfolds and for GM fourfolds satisfying condition 2 in Theorem \ref{thmK3U}. On the other hand, in Section 3.3 we see that a GM fourfold with period point of K3 type does not necessarily have a Hodge-associated K3 surface. 
\end{rmk}

\begin{proof}[Proof of Theorem \ref{thm3intro}]
The proof is analogous to that of \cite[Proposition 2.10]{Huy}. As done in \cite[Proposition 2.8]{Huy}, we have
$$\mathcal{D}_{\text{K3}'}=\Omega \cap \bigcup\limits_{\substack{0 \neq \varepsilon \in \tilde{\Lambda} \\\chi(\varepsilon,\varepsilon)=0}}
\varepsilon^{\perp}.$$
Assume that $x$ is a twisted K3 type point. By the previous observation, there exists an isotropic non trivial element $\varepsilon$ in $\tilde{\Lambda}$. We consider the lattice $\langle \lambda_1, \lambda_2, \varepsilon \rangle$ in $\tilde{\Lambda}$, with Euler pairing given by
$$
\begin{pmatrix}
-2 & 0 & x \\ 
0 & -2 & y \\ 
x & y & 0
\end{pmatrix}. 
$$
Notice that $\langle \lambda_1, \lambda_2, \varepsilon \rangle$ has discriminant $2x^2+2y^2$, which satisfies condition $(\ast\ast')$. Then, let $M_K$ be the saturation of $\langle \lambda_1, \lambda_2, \varepsilon \rangle$ in $\tilde{\Lambda}$. If $d$ is the discriminant of $M_K$ and $i$ is the index of $\langle \lambda_1, \lambda_2, \varepsilon \rangle$ in its saturation, then we have
$$2x^2+2y^2=i^2d.$$
It follows that also $d$ verifies condition $(\ast\ast')$, as we wanted. 

The other implication of the statement is proved following the same argument in the opposite direction.  
\end{proof} 

\subsection{Extending Theorem \ref{thmK3U}: a counterexample}

In this section we show that there are examples of GM fourfolds having a primitively embedded hyperbolic plane in the Néron-Severi lattice, but which cannot have a Hodge-associated K3 surface. Consistently with Theorem \ref{thmK3U}, our examples have $\text{rank}(N(\A_X))=4$ and their period points belong only to divisors corresponding to discriminants $\equiv 0 \pmod 8$.

Assume that $X$ is a GM fourfold such that $N(\A_X)=\langle \lambda_1,\lambda_2,\tau_1,\tau_2 \rangle$ with Euler form given by
\begin{equation}
\label{retcountex}
\begin{pmatrix}
 2 & 0 & 0 & 0 \\ 
 0 & 2 & 0 & 0 \\ 
 0 & 0 & -4 & -1-2n \\ 
 0 & 0 & -1-2n & -2(1+n^2)
 \end{pmatrix}
\end{equation} 
(here we consider the Mukai lattice $\KT(-1)$ with weight-two Hodge structure and quadratic form with reversed sign) with $n \in \N$. The matrix \eqref{retcountex} has signature $(2,2)$, compatibly with the Hodge Index Theorem. 

Notice that the classes
$$\kappa_1:=\lambda_1+\lambda_2+\tau_1 \quad \text{and} \quad \kappa_2:=\lambda_1+n\lambda_2+\tau_2$$
span a copy of $U$ in $N(\A_X)$.

However, it is easy to see that every labelling of $X$ will have discriminant congruent to $0$ modulo $8$; hence, we cannot find a labelling with discriminant satisfying $(\ast\ast)$. It follows that $X$ cannot have a Hodge-associated K3 surface.

It remains to prove that such a GM fourfold exists. The issue is that there is only a conjectural description of the image of the period map of GM fourfolds as the complement of some divisors in the period domain (see \cite[Question 9.1]{DIM}). As a consequence, the existence of GM fourfolds satisfying the above conditions is not a priori guaranteed. Anyway, we don't need such a strong result in order to prove that there is at least a value of $n$ and a GM fourfold with the required properties, as shown in the rest of this section. The key points are that we can reconstruct a GM fourfold from a given smooth double EPW sextic and that the union over $n$ of the sets of period points whose algebraic part contains a lattice as in \eqref{retcountex} is dense in a divisor with discriminant $16$.    

To this end, we firstly need the next lemma, where we study the conditions on $n$ in order to avoid the divisor $\mathcal{D}_8$. The motivation will be clear later, but essentially it comes from the fact that $\mathcal{D}_8$ contains period points of nodal GM fourfolds by \cite[Section 7.6]{DIM}, which we want to avoid as they are not smooth.

\begin{lemma}    
\label{no8}
The period point of a GM fourfold $X$ with $N(\A_X)$ as in \eqref{retcountex} is not in $\mathcal{D}_8$ if and only if $n \neq 0, 1$.
\end{lemma}
\begin{proof}
We actually prove that $p(X)$ is in $\mathcal{D}_8$ if and only if $n=0$ or $n=1$. A lattice spanned by $\lambda_1, \lambda_2, x \tau_1+y\tau_2$ has discriminant
$$Q(x,y):=4(-4x^2-2(1+n^2)y^2-2xy(1+2n))=-8(2x^2+(1+n^2)y^2+xy(1+2n)).$$
Notice that if $n=0$ (resp.\ $n=1$), then $Q(x,y)=-8$ for $(x,y)=(0,1)$ (resp.\ $(x,y)=\pm(1,-1)$. This implies that $p(X)$ is in $\mathcal{D}_8$.

On the other hand, assume $n>1$; then $-Q(x,y)/8 \geq 2x^2+5y^2+5xy$. The reduction of the latter term is $2x^2+xy+2y^2$, thus the smallest value it represents is $2$ (see \cite{Baker}, Section 5.2). We deduce that $-Q(x,y)/8$ does not represent $1$ for $n>1$. This implies the statement. 
\end{proof} 

We now prove that there is a value of $n \in \N$ as in Lemma \ref{no8} and a Hodge structure on $\tilde{\Lambda}$ having algebraic part given by the lattice in \eqref{retcountex}, defining a period point in the image of the period map of GM fourfolds. 

Let us introduce some notation. Fix $d:=-16$ and $n>1$. We denote by $\mathcal{S}_{n}$ the set of period points in $\mathcal{D}_d$ defined by the lattice $N_{n}$ as in \eqref{retcountex}. More precisely, this is the locus in $\mathcal{D}$ coming from $\p(N_{n}^{\perp} \otimes \C) \subset \p(\tilde{\Lambda} \otimes \C)$. We set
$$\mathcal{S}:=\bigcup_{n>1} \mathcal{S}_{n}$$
and we denote by $\mathcal{S}' \subset \mathcal{S}$ the locus of period points with algebraic part of rank-four. Thus, points in $\mathcal{S}'$ are very general points of $\mathcal{S}$ and their algebraic part is equal to a lattice $N_{n}$.

\begin{lemma}
\label{speriamosiavero}
The intersection of $\mathcal{S}'$ with the image of the period map $p$ is non empty.
\end{lemma}
\begin{proof}
Using the notation of \cite{DM}, let $\mathcal{M}_2^{(1)}$ be the moduli space of (smooth) hyperk\"ahler fourfolds deformation equivalent to the Hilbert square of a K3 surface, with polarization of degree-$2$ and divisibility $1$. Their period domain is $\mathcal{D}$. By \cite[Theorem 6.1 and Example 6.3]{DM}, the image of the period map $p_2^{(1)}: \mathcal{M}_2^{(1)} \hookrightarrow \mathcal{D}$ is equal to the complement of the divisors $\mathcal{D}_2$, $\mathcal{D}_8$ and $\mathcal{D}_{10}''$. As by definition and Lemma \ref{no8}, the period points we are considering are not in these divisors, it follows that $\mathcal{S}'$ is contained in the image of the period map $p_2^{(1)}$. 


We denote by $\mathcal{U}_2^{(1)}$ the Zariski open subset of $\mathcal{M}_2^{(1)}$ parametrizing smooth double EPW sextics. By \cite[Theorem 8.1]{DIM}, the Zariski open subset $p_2^{(1)}(\mathcal{U}_2^{(1)})$ of $\mathcal{D}$ meets $\mathcal{D}_d$. As a consequence, if we set
$$D_{2,d}^{(1)}:=\mathcal{D}_d \cap p_2^{(1)}(\mathcal{M}_2^{(1)}),$$
which is a hypersurface in $p_2^{(1)}(\mathcal{M}_2^{(1)})$, then
$$U^{(1)}_{2,d}:=D_{2,d}^{(1)} \cap p_2^{(1)}(\mathcal{U}_2^{(1)}) \neq \emptyset.$$
Moreover, $U^{(1)}_{2,d}$ is a Zariski open set in $D_{2,d}^{(1)}$.

We claim that
$$U^{(1)}_{2,d} \cap \bigcup_{n>1} \mathcal{S}_{n} \neq \emptyset.$$
Indeed, we set 
$$S_{n}:= \mathcal{S}_{n} \cap D_{2,d}^{(1)} \subset D_{2,d}^{(1)};$$
the union $\bigcup_{n>1} S_{n}$ of these Hodge loci is dense in $D_{2,d}^{(1)}$ (with respect to the Euclidean topology) by \cite[Proposition 5.20]{Voilibro}. As $U^{(1)}_{2,d}$ is Zariski open in $D_{2,d}^{(1)}$, we deduce the claim (the same argument is used in \cite{AHTV}, at the end of the proof of Theorem 1). Thus, there exists $n$ such that
$$U^{(1)}_{2,d} \cap \mathcal{S}_{n} \neq \emptyset.$$
As the set above is Zariski open in $\mathcal{S}_{n}$, it contains a very general point of $\mathcal{S}_{n}$, which belongs to $\mathcal{S}'$. It follows that 
$$p_2^{(1)}(\mathcal{U}_2^{(1)}) \cap \mathcal{S}' \neq \emptyset.$$

For every $x \in p_2^{(1)}(\mathcal{U}_2^{(1)}) \cap \mathcal{S}'$, we denote by $\tilde{Y}_A$ a smooth double EPW sextic such that $p_2^{(1)}([\tilde{Y}_A])=x$. Notice that the corresponding Lagrangian subspace $A$ has no decomposable vectors, because otherwise the period point of $\tilde{Y}_A$ would have been in $\mathcal{D}_8$ by \cite[Remark 5.29]{DK2}, which is not possible.

Finally, we observe that there exists a GM fourfold $X$ such that its associated double EPW sextic is precisely $\tilde{Y}_A$. Indeed, $\tilde{Y}_A$ determines a six-dimensional $\C$-vector space $V_6$ and a Lagrangian subspace $A \subset \bigwedge^3V_6$ without decomposable vectors. The choice of a five-dimensional subspace $V_5 \subset V_6$ with $A \cap \bigwedge^3V_5$ of the right dimension defines a Lagrangian data, which by \cite[Theorem 3.10, Proposition 3.13, and Theorem 3.16]{DK1} determines a GM fourfold $X$, as we wanted.   
\end{proof}

Applying Lemma \ref{speriamosiavero}, we deduce that there are GM fourfolds $X$ with $U \subset N(\A_X)$, but without a Hodge-associated K3 surface. 

\begin{rmk}
More generally, one can consider two integers $k$ and $l$ such that $(k,l) \neq (1,0), (0,1)$, $m, n \in \Z$ and the lattices $N_{k,l,m,n}:=\langle \lambda_1,\lambda_2,\tau_1,\tau_2 \rangle$ with Euler form given by
\begin{equation}
\label{retcountexgeneral}
\begin{pmatrix}
 2 & 0 & 0 & 0 \\ 
 0 & 2 & 0 & 0 \\ 
 0 & 0 & -2(k^2+l^2) & 1-2km-2ln \\ 
 0 & 0 & 1-2km-2ln & -2(m^2+n^2)
 \end{pmatrix}.
\end{equation}
It is possible to show that there are infinite values of $k, l, m, n$ such that the matrix $N_{k,l,m,n}$ has the right signature and determines a period point not in $\mathcal{D}_8$. Then, the argument in Lemma \ref{speriamosiavero} applies in this situation fixing $d=-8(k^2+l^2)$. Notice that the matrix in \eqref{retcountex} is precisely $N_{1,1,1,n}$. 

Again, the classes
$$\kappa_1:=k\lambda_1+l\lambda_2+\tau_1 \quad \text{and} \quad \kappa_2:=m\lambda_1+n\lambda_2+\tau_2$$
span a copy of $U$ in $N(\A_X)$. In the basis $\lambda_1,\lambda_2,\kappa_1,\kappa_2$, the Euler form is represented by
\begin{equation}
\label{matrixrmqpedante}
\begin{pmatrix}
 2 & 0 & 2k & 2m \\ 
 0 & 2 & 2l & 2n \\ 
 2k & 2l & 0 & 1 \\ 
 2m & 2n & 1 & 0
 \end{pmatrix}.
\end{equation}

Now, we claim that for every $\tau=a \kappa_1+b \kappa_2$ with $a, b \in \Z$, the lattice $\langle \lambda_1,\lambda_2,\tau \rangle$ has discriminant $\equiv 0 \pmod 8$. Indeed, it is enough to notice that $\chi(\tau,\tau)$, $\chi(\lambda_1,\tau)$ and $\chi(\lambda_2,\tau)$ are even. 

From the above observation, we deduce that the lattices $N_{k,l,m,n}$ as in \eqref{retcountexgeneral} represent all the possible intersection matrices of $N(\A_X)$ in the rank-four case that do not satisfy the condition in item 2 of Theorem \ref{thmK3U}. Indeed, there is not a class in the embedded $U$ which gives rise to a labelling of discriminant $\equiv 2, 4 \pmod 8$.
\end{rmk}

\begin{rmk}
We have proved that there are divisors of the form $\mathcal{D}_{8t}$, whose elements cannot have a Hodge-associated K3 surface as $(\ast \ast)$ does not hold, but containing period points of GM fourfolds $X$ with $U \subset N(\A_X)$. In particular, the latter property does not hold for all the elements in $\mathcal{D}_{8t}$. It follows that the condition of having an embedded $U$ in $N(\A_X)$ is not divisorial, in contrast to what happens for cubic fourfolds.
\end{rmk}

\section{Associated double EPW sextic}

The aim of this section is to prove Theorems \ref{thm1}, \ref{thm4} and \ref{thm2} stated in the introduction. We follow the argument of \cite{Add} and of \cite{Huy} for the twisted case; in particular, we define a Markman embedding for $H^2(\tilde{Y}_A,\Z)$ in $\tilde{\Lambda}$ and we apply Propositions 4 and 5 of \cite{Add}. 

\subsection{Proof of Theorem \ref{thm1} and \ref{thm4}}

Assume that $X$ is a GM fourfold with Lagrangian data $(V_6,V_5,A)$ such that $\tilde{Y}_A$ is smooth. Before starting with the proofs, we need the following lemma, which relates the sublattice $\langle \lambda_1 \rangle^{\perp}$ of $\KT$ (equipped with the induced Hodge structure) and $H^2(\tilde{Y}_A,\Z)$.   

\begin{lemma}
\label{mukaivsprim}
There exists an isometry of Hodge structures between the lattices $\langle \lambda_1 \rangle^{\perp} \subset K(\A_X)_{\emph{top}}$ and $H^2(\tilde{Y}_A,\Z)(1)$.
\end{lemma}
\begin{proof}
Composing the isometry of Proposition \ref{mukaivsvan} with that of Theorem \ref{periodEPW}, we obtain the Hodge isometry
$$f: \langle \lambda_1,\lambda_2 \rangle^{\perp} \cong H^2(\tilde{Y}_A,\Z)_0(1).$$
Notice that twisting by $1$, we have shifted by two the weight of the Hodge structure on the primitive cohomology and we have reversed the sign of $q$; in particular, $f$ is an isometry of weight zero Hodge structures.  

Now, we observe that $\langle \lambda_1 \rangle^{\perp}$ is isomorphic to $E_8^2 \oplus U^3 \oplus \Z(u_1+v_1)$ via the marking $\phi$ defined in \eqref{isophi}. On the other hand, by \cite[Proposition 6]{Beau}, we have the isometry 
$$H^2(\tilde{Y}_A,\Z) \cong H^2(S,\Z) \oplus \Z\delta \cong E_8(-1)^2 \oplus U^3 \oplus I_{0,1}(2),$$ 
where $S$ is a degree-two K3 surface and $q(\delta)=-2$. Twisting by $1$, we get
$$\psi: H^2(\tilde{Y}_A,\Z)(1) \cong E_8^2 \oplus U^3 \oplus \Z\delta, \quad \text{with }q(\delta)=2,$$ 
using that $U(-1)\cong U$. In particular, $\langle \lambda_1 \rangle^{\perp}$ and $H^2(\tilde{Y}_A,\Z)(1)$ are isomorphic lattices. 

Let $h_A$ be the polarization class on $\tilde{Y}_A$ with satisfies $q(h_A)=-2$ in $H^2(\tilde{Y}_A,\Z)(1)$ (see \cite{OG2}, eq.\ (1.3)). 
We define an isometry $g: \langle \lambda_1, \lambda_2 \rangle^{\perp} \oplus \langle \lambda_2 \rangle \cong H^2(\tilde{Y}_A,\Z)_0(1) \oplus \langle h_A \rangle $ such that 
$$g(\lambda_2)=h_A \quad  \text{and} \quad g(\langle \lambda_1,\lambda_2 \rangle^{\perp})= f(\langle \lambda_1,\lambda_2 \rangle^{\perp})=H^2(\tilde{Y}_A,\Z)_0(1).$$
Notice that $g$ preserves the Hodge structures, because $f$ does and $g$ sends the $(0,0)$ class $\lambda_2$ to the $(0,0)$ class $h_A$. In particular, $g$ defines an isomorphism of Hodge structures $\langle \lambda_1 \rangle^{\perp} \cong H^2(\tilde{Y}_A,\Q)(1)$ over $\Q$. 

We claim that $g$ extends to an isometry $\langle \lambda_1 \rangle^{\perp} \cong H^2(\tilde{Y}_A,\Z)(1)$ over $\Z$. Indeed, we set $S_1:=H^2(\tilde{Y}_A,\Z)_0(1)$, $S_2:=\langle \lambda_1,\lambda_2 \rangle^{\perp}$ and $L:=H^2(\tilde{Y}_A,\Z)(1)$. We denote by $K_1$ and $K_2$ the orthogonal complements of $S_1$ and $S_2$ in $L$. By definition, we have  $K_1\cong \langle h_A \rangle$ and $K_2 \cong \langle \lambda_2 \rangle$. 
We set 
$$H_1:=\frac{L}{S_1 \oplus K_1} \subset d(S_1) \oplus d(K_1) \quad \text{and} \quad H_2:=\frac{L}{S_2 \oplus K_2} \subset d(S_2) \oplus d(K_2);$$
recall that
$$d(S_i) \cong \frac{\Z}{2\mathbb{Z}} \oplus \frac{\Z}{2\mathbb{Z}} \quad \text{ and } \quad d(K_i) \cong \frac{\Z}{2\mathbb{Z}}.$$
Let $H_{i,S}$ and $H_{i,K}$ be the projections of $H_i$ in $d(S_i)$ and $d(K_i)$, respectively. Then, there is an isomorphism $\gamma_i: H_{i,S} \cong H_{i,K}$, given by the composition of the inverse of the projection on the first factor with the projection to the second factor. By definition $H_{i,K}$ is a subgroup of $d(K_i) \cong \Z/2\mathbb{Z}$. We exclude the case $H_{i,K}=0$. Then we have $H_{i,K}=d(K_i)$. We list the generators of the subgroups of $d(S_i) \oplus d(K_i)$ mapping to $\Z/2\mathbb{Z}$ via the two projections:
$$(1,0,1), (0,1,1) \text{ and } (1,1,1).$$ 
Since $H_i$ is isotropic with respect to $q:=q_{S_i} \oplus q_{K_i}$, we exclude $(1,1,1)$, because
$$q((1,1,1))= \frac{1}{2}+\frac{1}{2}-\frac{1}{2}=-\frac{1}{2} \neq 0 \mod \mathbb{Z}.$$
Moreover, recall that by \cite[Proposition 1.4.1(b)]{Ni}, we have 
$$d(L)\cong \frac{H_i^{\perp}}{H_i},$$
where $H_i^{\perp}$ is the orthogonal to $H_i$ in $d(S_i) \oplus d(K_i)$. This condition implies that $H_i=\langle (0,1,1) \rangle$. Indeed, assume that $H_i=\langle (1,0,1) \rangle$. Writing explicitely the generators of the discriminant groups we have
$$d(K_i)=\langle \frac{f_2}{2} \rangle, \quad d(S_i)=\langle \frac{g_1}{2},\frac{g_2}{2} \rangle, \quad d(L)=\langle \frac{g_1}{2} \rangle, \quad H_i=\langle \frac{g_1}{2}+\frac{f_2}{2} \rangle.$$
However, we have
$$\frac{H_i^{\perp}}{H_i}=\frac{\langle \frac{g_1}{2}+\frac{f_2}{2},\frac{g_2}{2} \rangle}{\langle \frac{g_1}{2}+\frac{f_2}{2} \rangle}=\langle \frac{g_2}{2} \rangle$$
giving a contradiction.

Now, recall that by \cite[Corollary 1.5.2]{Ni}, the isometry $f$ extends to an isometry of $L$ if and only if there exists an isometry $f': K_1 \to K_2$ such that the diagram
\[
\xymatrix{
d(S_1)& \ar[l]_{\supset} H_{1,S} \ar[d]^{\bar{f}} \ar[r]^-{\gamma_1}_{\cong} & H_{1,K} \ar[d]^{\bar{f'}} \ar[r]^-{=}  & d(K_1) \\
d(S_2)& \ar[l]_{\supset} H_{2,S} \ar[r]^-{\gamma_2}_{\cong} & H_{2,K} \ar[r]^-{=}  &d(K_2).}
\]
commutes, where $\bar{f}$ and $\bar{f'}$ are induced by $f$ and $f'$ on the discriminant groups.
So, we consider the isometry $f': K_1 \to K_2$ sending $h_A$ to $\lambda_2$; notice that $f'$ acts trivially on the discriminant group. On the other hand, the isometry $f$ acts either as the identity on $\mathbb{Z}/2\mathbb{Z} \times \mathbb{Z}/2\mathbb{Z}$ or it exchanges the two factors. Assume we are in the first case. Then, it follows that $\bar{f'} \circ \gamma_1 ((0,1))=\gamma_2 \circ \bar{f}((0,1))$. 

In the second case, we change the marking $\phi$ with the marking $\phi': \KT \cong \tilde{\Lambda}(-1)$, such that $\phi'(\lambda_1)=f_2$ and $\phi'(\lambda_2)=f_1$. By the same argument explained above, we have $H_2=\langle (1,0,1) \rangle$ and $H_{2,S}=\langle (1,0) \rangle$. It follows that
$$\gamma_2 \circ \bar{f}((0,1))= \gamma_2((1,0))=H_{2,K}=\bar{f'} \circ \gamma_1 ((0,1)).$$
Then \cite[Corollary 1.5.2]{Ni} applies and we deduce that the isometry $f$ extends to an isometry of $L$. It follows that $g$ is well defined over $\Z$, which concludes the proof. 
\end{proof}

\begin{rmk}
\label{Memb2}
In the same way we can prove that there is a Hodge isometry $\langle \lambda_2 \rangle^{\perp} \cong H^2(\tilde{Y}_A,\Z)(1)$ which extends $f$ and sending $\lambda_1$ to $h_A$.
\end{rmk}

By Lemma \ref{mukaivsprim}, it follows that there is a primitive embedding 
$$H^2(\tilde{Y}_A,\Z) \hookrightarrow \KT(-1).$$ 
By \cite[Section 9]{Mark}, it is unique up to isometry of $\tilde{\Lambda}$. Thus it defines a \textbf{Markman embedding} as discussed in \cite[Section 1]{Add}. 
\begin{proof}[Proof of Theorem \ref{thm1}]
If $d$ satisfies $(\ast\ast)$, then $N(\A_X)$ contains a copy of the hyperbolic plane $U$ by Theorem \ref{thmK3U}. This proves that (a) implies (b). Recall that $T_X$ is Hodge isometric to $T_{\tilde{Y}_A}(-1)$ by Theorem \ref{periodEPW}. Then (b) is equivalent to (c) by Proposition 4 in \cite{Add}. 

Assume that $X$ is as in the second part of the statement. Then by Theorem \ref{thmK3U} $d$ satisfies $(\ast\ast)$ if and only if $U \subset N(\A_X)$. The statement follows applying \cite[Proposition 4]{Add} as before. 
\end{proof}

\begin{rmk}
As observed in \cite{Add} for cubic fourfolds, under the hypothesis of Theorem \ref{thm1}, $\tilde{Y}_A$ is birational to a moduli space of Bridgeland stable objects if and only if $d$ satisfies $(\ast\ast)$, by \cite[Theorem 1.2(c)]{BM}.
\end{rmk}

\begin{rmk}
As observed in \cite[Remark 5.29]{DK2}, the image of the closure of the locus of smooth GM fourfolds having singular associated double EPW sextic is precisely the divisor $\mathcal{D}_{10}''$. We claim that there exist Hodge-special GM fourfolds with smooth associated double EPW sextic. Indeed, by \cite[Section 7.2]{DIM} this is clear for general GM fourfolds containing a $\tau$-plane: their period points lie in the divisor $\mathcal{D}_{12}$ and they do not belong to $\mathcal{D}_{10}''$, because of generality assumption. Now, let $d$ be a positive integer $\equiv 0,2,4 \pmod 8$. Assume $d >12$ if $d \equiv 0 \pmod 4$, resp.\ $d\geq 10$ if $d\equiv 2 \pmod 8$. By \cite[Theorem 8.1]{DIM}, the image of the period map meets all divisors $\mathcal{D}_d$, $\mathcal{D}_d'$ and $\mathcal{D}_d''$ for the respective values of the discriminant. More precisely, for every $d$ as before, they construct a GM fourfold $X_0$ whose period point $p(X_0)$ belongs to the intersection of $\mathcal{D}_{10}''$ with $\mathcal{D}_d$ (resp.\ $\mathcal{D}_d'$ or $\mathcal{D}_d''$) if $d \equiv 0 \pmod 4$ (resp.\ $d \equiv 2 \pmod 8$). Consider the case $d \equiv 0 \pmod 4$. Since the period map is dominant (see Section 2.2), there exists an open dense subset $U$ of $\mathcal{D}$ contaning $p(X_0)$ such that $U \subseteq p(\mathcal{M}_4)$. Notice that $U \cap \mathcal{D}_d$ is open in $\mathcal{D}_d$ and it contains $p(X_0)$. Moreover, it is not contained in $\mathcal{D}_d \cap \mathcal{D}_{10}''$, because the latter has codimension $1$ in $\mathcal{D}_d$. It follows that $(U \cap \mathcal{D}_d)\setminus \mathcal{D}_{10}'' \neq \emptyset$. The same argument applies to the case $d \equiv 2 \pmod 8$ and it completes the proof of the claim.
\end{rmk}

\begin{rmk}
Assume that $X$ is a Hodge-special GM fourfold such that $\tilde{Y}_A$ is smooth. Notice that the period point defined by the Hodge structure on $K(\A_X)_{\text{top}}(-1)$ is of K3 type if and only if $\tilde{Y}_A$ is birational to a moduli space of stable sheaves on a K3 surface. 
\end{rmk}

As in \cite[Proposition 4.1]{Huy} in the case of cubic fourfolds, we can prove the twisted version of Theorem \ref{thm1}.

\begin{proof}[Proof of Theorem \ref{thm4}]
Assume that $\tilde{Y}_A$ is birational to a moduli space $M(v)$ of $\alpha$-twisted stable sheaves on a K3 surface $S$, where $v$ is primitive in $\tilde{H}^{1,1}(S,\alpha,\Z)$ and $(v,v)=2$. Using the embedding $H^2(\tilde{Y}_A,\Z) \hookrightarrow \KT(-1)$ and Torelli Theorem for hyperk\"ahler manifolds, this is equivalent to have an isometry of Hodge structures $\KT(-1) \cong \tilde{H}(S,\alpha,\Z)$, which restricts to 
$$H^2(\tilde{Y}_A,\Z) \cong H^2(M(v),\Z)\cong v^{\perp} \hookrightarrow \tilde{H}(S,\alpha,\Z).$$
Equivalently, by Theorem \ref{thm3intro}, $X$ is Hodge-special with a labelling of discriminant $d$ satisfying condition $(\ast\ast')$. This proves one direction.

On the other hand, assume that $p(X)$ belongs to a divisor with $d$ satisfying $(\ast\ast')$. Then the image $v$ of $\lambda_1$ through $H^2(\tilde{Y}_A,\Z) \hookrightarrow \KT(-1) \cong \tilde{H}(S,\alpha,\Z)$ is primitive. Since the induced moduli space $M(v)$ is non-empty and the Hodge isometry $H^2(\tilde{Y}_A,\Z)\cong v^{\perp} \cong H^2(M(v),\Z)$ extends to $\tilde{\Lambda}$, we conclude the desired statement. 
\end{proof}

\subsection{Proof of Theorem \ref{thm2}}

Firstly, we need the following lemma, which is analogous to \cite[Lemma 9]{Add} and that we will use also in Section 4.2.

\begin{lemma}
\label{lemmaformaeulfacile}
Let $X$ be a Hodge-special GM fourfold of discriminant $d$ such that $d \equiv 2\text{ or }4 \pmod 8$. Then there exists an element $\tilde{\tau}$ in $N(\A_X)$ such that $\langle \lambda_1,\lambda_2,\tilde{\tau} \rangle$ is a primitive sublattice of discriminant $d$ with Euler pairing given, respectively, by
\[
\begin{pmatrix}
-2 & 0 & 1 \\ 
0 & -2 & 0 \\ 
1 & 0 & 2k
\end{pmatrix} \quad \text{or} \quad 
\begin{pmatrix}
-2 & 0 & 0 \\ 
0 & -2 & 1 \\ 
0 & 1 & 2k
\end{pmatrix} \quad \text{with }d=2+8k,
\]
\[
\begin{pmatrix}
-2 & 0 & 1 \\ 
0 & -2 & 1 \\ 
1 & 1 & 2k
\end{pmatrix} \quad \text{with }d=4+8k.
\]
\end{lemma}
\begin{proof}
By Corollary \ref{correticoli}, there exists an element $\tau \in N(\A_X)$ such that $\langle \lambda_1,\lambda_2,\tau \rangle$ is a primitive sublattice of discriminant $d$ with Euler paring given by
\[
\begin{pmatrix}
-2 & 0 & a \\ 
0 & -2 & b \\ 
a & b & c
\end{pmatrix}.\]
Notice that $c$ is even, because $N(\A_X)$ is an even lattice. 

Assume that $d \equiv 2 \pmod 8$; then one of $a$ and $b$ is even and the other is odd. Assume that $b$ is even. Substituting $\tau$ with $\tau'=\tau +b/2 \lambda_2$, we get a new basis with Euler form
\[
\begin{pmatrix}
-2 & 0 & a \\ 
0 & -2 & 0 \\ 
a & 0 & c'
\end{pmatrix}.
\]
We can write $a=4d+e$ with $e=-1,1$. Then, the basis $\lambda_1,\lambda_2, \tilde{\tau}=\tau'+2d\lambda_1$ has Euler pairing
\[
\begin{pmatrix}
-2 & 0 & e \\ 
0 & -2 & 0 \\ 
e & 0 & 2k
\end{pmatrix}.
\]
If $e=-1$, we change $\tilde{\tau}$ with $-\tilde{\tau}$ and we return to the case $e=1$. If $a$ is even, by the same argument we obtain a basis with the second matrix in the statement. This proves the claim in the case $d \equiv 2 \pmod 8$. The case $d \equiv 4 \pmod 8$ works in the same way. 
\end{proof}

\begin{proof}[Proof of Theorem \ref{thm2}]
Assume that there exist a K3 surface $S$ and a birational equivalence $\tilde{Y}_A \dashrightarrow S^{[2]}$. By Lemma \ref{mukaivsprim} and \cite[Proposition 5]{Add}, this is equivalent to the existence of an element $w$ in $N(\A_X)$ such that the Euler pairing of $K:=\langle\lambda_1,\lambda_2,w \rangle$ has the form
\[
\begin{pmatrix}
-2 & 0 & 1 \\ 
0 & -2 & n \\ 
1 & n & 0
\end{pmatrix} \quad \text{for some }n \in \Z.
\]
In particular, the discriminant of $K$ is $2n^2+2$. Let $M_K$ be the saturation of $K$; if $a$ is the index of $K$ in $M_K$ and $d$ is the discriminant of $M_K$, we have $\text{discr}(K)=a^2d$, as we wanted.

Conversely, assume that $d$ satisfies condition \eqref{ast3}. Then there exist integers $n$ and $a$ such that $a^2d=2n^2+2$. Firstly, we observe that $2n^2+2$ satisfies $(\ast\ast)$. Indeed, every odd prime $p$ dividing $n^2+1$ is $\equiv 1 \pmod 4$, and $8 \nmid  2n^2+2$. It follows that $a$ is the product of odd primes $\equiv 1 \pmod 4$; in particular, $a \equiv 1 \pmod 4$. 

Suppose firstly that $d  \equiv 2 \pmod 8$; then $n$ is even. Indeed, assume that $n \equiv 1 \pmod 4$ (resp.\ $n \equiv 3 \pmod 4$). It follows that $n^2+1 \equiv 2 \pmod 4$; then $d \equiv 4 \pmod 8$, which is impossible. Furthermore, by Lemma \ref{lemmaformaeulfacile}, there is an element $\tau \in N(\A_X)$ such that the sublattice $\langle \lambda_1,\lambda_2,\tau \rangle$ has Euler pairing of the form
\[
\begin{pmatrix}
-2 & 0 & 1 \\ 
0 & -2 & 0 \\ 
1 & 0 & 2k
\end{pmatrix}
\quad \text{or}\quad
\begin{pmatrix}
-2 & 0 & 0 \\ 
0 & -2 & 1 \\ 
0 & 1 & 2k
\end{pmatrix}.
\] 

Assume that we are in the first case. We set
$$w:=\frac{a-1}{2} \lambda_1 + \frac{n}{2} \lambda_2 + a\tau \in N(\A_X),$$
where $n/2$ is an integer, because $n$ is even. An easy computation shows that 
$$\chi(\lambda_1,w)=1 \quad \text{and} \quad \chi(w)=0.$$
By \cite[Proposition 5]{Add}, it follows that $\tilde{Y}_A$ is birational to $S^{[2]}$ for a K3 surface $S$. 

If we are in the second case, we consider the Markman embedding $H^2(\tilde{Y}_A,\Z) \subset \KT(-1)$ defined by the Hodge isometry $\langle \lambda_2 \rangle^{\perp} \cong H^2(\tilde{Y}_A,\Z)(1)$ (see Remark \ref{Memb2}). Setting
$$w:=\frac{n}{2} \lambda_1 + \frac{a-1}{2} \lambda_2 + a\tau \in N(\A_X),$$
the proof follows from \cite[Proposition 5]{Add}.

Now assume that $d \equiv 4 \pmod 8$; then $n$ is odd. Indeed, if $n \equiv 0 \pmod 4$ (resp.\ $n \equiv 2 \pmod 4$), then $n^2+1 \equiv 1 \pmod 4$. Since $a^2d/2=n^2+1$ and $a \equiv 1 \pmod 4$, we conclude that $d/2 \equiv 1 \pmod 4$, which is impossible. By Lemma \ref{lemmaformaeulfacile}, there is an element $\tau \in N(\A_X)$ such that the sublattice $\langle \lambda_1,\lambda_2,\tau \rangle$ has Euler pairing of the form
\[
\begin{pmatrix}
-2 & 0 & 1 \\ 
0 & -2 & 1 \\ 
1 & 1 & 2k
\end{pmatrix} \text{ with }d=4+8k.
\]
We set
$$w:=\frac{a-1}{2} \lambda_1 + \frac{a-n}{2} \lambda_2 + a\tau \in N(\A_X).$$
Notice that $(a-n)/2$ is integral, because $n$ is odd. Arguing as before, we conclude the proof of the result.
\end{proof}

\begin{rmk}
\label{d=50}
As seen in the proof of Theorem \ref{thm2}, condition \eqref{ast3} implies condition $(\ast\ast)$. On the other hand, condition \eqref{ast3} is stricter than condition $(\ast\ast)$. For example, $d=50$ satisfies $(\ast\ast)$, but not \eqref{ast3}.
\end{rmk}

\begin{rmk}
\label{linkIM}
In \cite[Proposition 2.1]{IM}, they proved that if a smooth double EPW sextic is birational to the Hilbert scheme $S^{[2]}$ on a K3 surface $S$ with polarization of the degree-$d$, then the negative Pell equation
$$\mathcal{P}_{d/2}(-1): n^2-\frac{d}{2}a^2=-1$$
is solvable in $\Z$. This condition is actually condition \eqref{ast3} in the case of the double EPW associated to a GM fourfold. Notice also that the K3 surface $S$, realizing the birational equivalence between $\tilde{Y}_A$ and $S^{[2]}$ in Theorem \ref{thm2}, has a pseudo-polarization of degree-$d$. Indeed, the hypothesis implies that there is a rank-three sublattice $M_K$ of $N(\A_X)$. Moreover, it contains a copy of the hyperbolic plane and $H^2(S,\Z) \cong U^{\perp} \subset N(\A_X)$, as explained in the proof of \cite[Proposition 5]{Add}. Then, the generator of $U^{\perp} \cap M_K$ has degree-$d$, as we wanted. Moreover, if $p(X) \notin \mathcal{D}_8$, then there are no classes of square $2$ in $H^4(X,\Z)_{00} \cap H^{2,2}(X,\Z)$. In this case, the pseudo-polarization is a polarization class on $S$. 


\end{rmk}

Dipartimento di Matematica ``F.\ Enriques'', Universit\`a degli Studi di Milano, Via Cesare Saldini 50, 20133 Milano, Italy \\
\indent E-mail address: \texttt{laura.pertusi@unimi.it}\\
\indent URL: \texttt{http://www.mat.unimi.it/users/pertusi}

\end{document}